\documentclass[twoside]{article}
\usepackage[left=4cm,right=4cm,
top=3cm,bottom=2.5cm,bindingoffset=0cm]{geometry}
\usepackage{graphicx, hyperref}
\usepackage[french]{babel}
\usepackage[T1]{fontenc}
\usepackage{amsmath,amscd, wasysym, mathtools,array, mathtools, mathrsfs}
\usepackage[all]{xy}
\usepackage{upgreek}

\usepackage[sc, tiny, center]{titlesec}

\titleformat{\subsection}    
    {\normalfont\bfseries}{\thesubsection}{1em}{}

\usepackage{eufrak}
\usepackage{hyperref} 
\usepackage{amssymb,amsthm,euscript}
\usepackage{xcolor,graphicx,epsfig,subfigure}
\usepackage{dsfont}
\usepackage{enumitem}

\usepackage{amsfonts}
\usepackage{fancyhdr}
\usepackage{hyperref}

\newtheorem{theorem}{Théorème}[section]
\newtheorem*{theorem*}{Théorème}
\newtheorem{lemma}[theorem]{Lemme}
\newtheorem*{lemma*}{Lemme}
\newtheorem{proposition}[theorem]{Proposition}

\newtheorem{corollary}[theorem]{Corollaire}

\theoremstyle{definition}
\newtheorem{definition}[theorem]{Définition}

\newtheorem{remark}[theorem]{Remarque}
\newtheorem{example}[theorem]{Exemple}
\usepackage{chngcntr}

\renewenvironment{proof}{
\par\noindent{\it Démonstration.}} {\mbox{}\hfill$\square$ \par }

\newcommand{\into}{\hookrightarrow}

\newcommand{\bbP}{\mathbb{P}}
\newcommand{\bbA}{\mathbb{A}}
\newcommand{\bbG}{\mathbb{G}}
\newcommand{\bbC}{\mathbb{C}}
\newcommand{\bbZ}{\mathbb{Z}}
\newcommand{\bbQ}{\mathbb{Q}}
\newcommand{\bbR}{\mathbb{R}}

\newcommand {\Gal}{\mathord{\rm{Gal}}}

\newcommand{\Mat}{\mathord{\rm{Mat}}}
\newcommand{\Gr}{\mathord{\rm{Gr}}}
\newcommand{\Id}{\mathord{\rm{Id}}}
\newcommand{\Br}{\mathord{\rm{Br}}}
\newcommand{\Hom}{\mathord{\rm{Hom}}}
\newcommand{\Sym}{\mathord{\rm{Sym}}}
\newcommand{\Frac}{\mathord{\rm{Frac}}}

\newcommand{\res}{\mathord{\rm{res}}}
\newcommand{\lisse}{\mathord{\rm{lisse}}}

\newcommand{\inv}{\mathord{\rm{inv}}}

\newcommand{\car}{\mathord{\rm{car}}}
\newcommand{\sgn}{\mathord{\rm{sgn}}}
\newcommand{\disc}{\mathord{\rm{disc}}}
\newcommand{\cliff}{\mathord{\rm{cliff}}}
\newcommand{\sing}{\mathord{\rm{sing}}}
\newcommand{\rk}{\mathord{\rm{rk}}}

\newcommand{\overbar}[1]{\mkern 1.5mu\overline{\mkern-1.5mu#1\mkern-1.5mu}\mkern 1.5mu}

\usepackage{lipsum}

\DeclareFontFamily{U}{wncy}{}
    \DeclareFontShape{U}{wncy}{m}{n}{<->wncyr10}{}
    \DeclareSymbolFont{mcy}{U}{wncy}{m}{n}
    \DeclareMathSymbol{\Sh}{\mathord}{mcy}{"58}
    
\title{\Large{\scshape{\textbf{le principe de hasse pour les intersections\\de deux quadriques dans $\mathbb{P}^7$}}}}
\date{}
\author{Alexander Molyakov}

\begin{document}

\maketitle
\begin{abstract}
On démontre le principe de Hasse pour un modèle propre lisse d'une intersection géométriquement intègre non conique de deux quadriques dans l'espace projectif de~dimension 7 sur un corps de nombres. Ce résultat généralise le résultat de Heath-Brown \cite{HB} qui a établi l'énoncé dans le cas lisse. Notre argument est basé sur les idées de Colliot-Thélène \cite{CT} et la méthode des fibrations pour les zéro-cycles dans la forme de Harpaz-Wittenberg \cite{HW}.\\

We prove the Hasse principle for a smooth proper model of a geometrically integral non-conical intersection of two quadrics in the projective space of dimension~7 over a number field. This result generalizes the result of Heath-Brown \cite{HB} who established the statement in the smooth case. Our argument is built upon the ideas of \linebreak Colliot-Thélène \cite{CT} and the fibration method for zero-cycles in the form of Harpaz-Wittenberg~\cite{HW}.
\end{abstract}

\section*{Introduction}

L'étude des principes locaux-globaux pour les points rationnels sur une variété algébrique est un problème classique de la géométrie arithmétique. Une variété lisse sur un corps de nombres $k$ satisfait le \textit{principe de Hasse} si l'existence locale d'un point rationnel sur chaque complété $k_v$ implique l'existence d'un $k$-point rationnel. Dans le cas des variétés singulières il est raisonnable de modifier cette définition en exigeant que tout modèle propre lisse de la variété satisfasse le principe de Hasse. Cette modification s'appelle le \textit{principe de Hasse lisse}. En même temps, il existe beaucoup d'exemples de variétés qui ne satisfont pas le principe de Hasse. \textit{L'obstruction de Brauer-Manin} est une obstruction au~principe de Hasse qui a été introduite par Y. Manin \cite{Mn} en termes du groupe de Brauer étale de la variété.  Une conjecture générale formulée par J.-L. Colliot-Thélène \cite{CT03} affirme que l'obstruction de Brauer-Manin au principe de Hasse est la seule pour les variétés rationnellement connexes. Cette conjecture reste largement ouverte même dans le~cas des variétés géométriquement rationnelles. 

Les intersections géométriquement intègres et non coniques de deux quadriques dans l'espace projectif $\bbP^n\;(n\geqslant 4)$ fournissent l'un des premiers exemples non triviaux de~variétés géométriquement rationnelles. D'après la conjecture de Colliot-Thélène une telle variété $X\subset \bbP^n$ doit satisfaire le principe de Hasse lisse à partir de la dimension $n=6$ puisque le groupe de Brauer relatif $\Br(\hat X)/\Br(k)$ d'un modèle propre lisse $\hat X$ de $X$ s'annule selon le calcul fait dans \cite{CSS}. Pour les mod\`eles propres et lisses de ces vari\'et\'es on sait aussi que  le principe
de Hasse implique l'approximation faible quand $n \geqslant 6$ \cite[Theorem 3.11]{CSS}. Dans l'article \cite{CSS} Colliot-Thélène, Sansuc et Swinnerton-Dyer ont montré le principe de Hasse lisse sans restrictions dans la dimension  $n\geqslant 8$ et aussi pour certains types d'intersections spéciales dans la dimension $n\geqslant 6$, en particulier, celles qui contiennent une paire de droites conjuguées ou une paire de points singuliers conjugués. Trente ans plus tard, D.R. Heath-Brown a établi le principe de Hasse pour toute intersection \textit{lisse} dans la dimension $n=7$ \cite{HB}. On généralise ce résultat en montrant que le principe de Hasse lisse vaut pour une intersection quelconque (possiblement singulière) de deux quadriques dans $\bbP^7$. Plus précisément, on prouve le~théorème suivant indiqué comme le Théorème \ref{theoremeprincipal} ci-dessous.
\begin{theorem*}
Soit $k$ un corps de nombres. Soit $X\subset \bbP^7_k$ une intersection complète géométriquement intègre non conique de deux quadriques. Alors tout
modèle propre lisse de $X$ satisfait le principe de Hasse.     \end{theorem*}

Notamment, l'argument donné dans \cite{HB} pour $X$ lisse utilise un énoncé local sur l'existence d'une conique contenue dans $X$ sur $k_v$ quand $X(k_v)\neq \varnothing$. J.-L. Colliot-Thélène a~récemment revisité le résultat de Heath-Brown et a trouvé une démonstration relativement courte du principe de Hasse dans le cas lisse. Sa démonstration utilise un théorème de Salberger \cite{Sb} et s'appuie également sur l'existence locale d'une conique qui est déduite de \cite[Proposition 1.6]{CT} formant donc la base de l'argument. Cette dernière proposition affirme l'existence d'un point quadratique sur une intersection de deux quadriques dans $\bbP^3$ sur un corps $p$-adique dont le pinceau contient une forme dégénérée définie sur le corps de base. Notre approche suit celle de Colliot-Thélène pour traiter le cas où $X$ est \textit{régulière} dans le sens de \cite{Wn}. Cependant, l'énoncé local correspondant (la Proposition \ref{prop8} ci-dessous) exige du travail supplémentaire afin de trouver des sections de $X$ suffisamment génériques. Pour compléter la démonstration il faut montrer que $X_{k_v}$ contient un plan pour presque toute place $v$ de $k$. On peut appliquer la description explicite de la variété $F$ paramétrant les plans contenus dans le lieu lisse de $X$ selon Wang \cite{Wn}. Dans l'appendice on présente une autre façon d'obtenir ce résultat (Proposition \ref{propplan}).  

L'autre cas qu'on appelle \textit{irrégulier} est assez différent. Tandis que dans bien des cas on peut toujours trouver localement une conique par la Proposition \ref{prop8}, $X_{k_v}$ peut ne pas contenir un plan pour un nombre infini de places $v$ (Exemple \ref{exemplenonreg}). C'est la raison pour laquelle on a besoin d'une autre méthode. On~divise le~cas irrégulier en trois possibilités en fonction du rang des formes quadratiques dégénérées dans le pinceau géométrique de $X$, voir le début de la section 3. Quand le~pinceau géométrique contient une forme de rang $\leqslant 5$ on réduit le problème au résultat de \cite{CSS} sur les intersections contenant une paire de droites conjuguées (Proposition \ref{cas5}). L'idée de la démonstration dans les cas qui restent est basée sur la méthode des fibrations pour les zéro-cycles dans la forme de \cite{HW} vu l'équivalence entre l'existence d'un zéro-cycle de degré 1 et d'un point rationnel sur une intersection de deux quadriques qui résulte du théorème de Amer-Brumer (Théorème \ref{prop16}).  On déduit de cette équivalence l'équivalence entre plusieurs principes locaux-globaux pour les intersections de deux quadriques (Proposition \ref{core}) y inclus la conjecture \hyperref[(E)]{(E)}. Ensuite, on construit une fibration de $X$ en intersections singulières de deux quadriques dans $\bbP^4$ ou en espaces homogènes principaux de tores algébriques pour lesquels il est connu que l'obstruction de Brauer-Manin à l'approximation faible est la seule \cite{CorTs, CSS, Sn} et, donc, qu'ils satisfont la conjecture \hyperref[(E)]{(E)} d'après le résultat général de Liang \cite{L}. Cela nous permet de conclure que $X$ satisfait la conjecture \hyperref[(E)]{(E)} qui assure le principe de Hasse lisse étant donné l'équivalence mentionnée ci-dessus.

Dans l'introduction de \cite{HB} l'auteur écrit: "\textit{While Theorem 2 can be appropriately extended to singular intersections, it is unclear 
whether one can prove a corresponding global statement. Thus the methods of the present paper 
seem insufficient to handle a version of Theorem 1 for singular intersections of two quadrics in $\bbP^7$}." De façon surprenante, le théorème local n'est pas étendu en général aux intersections
singulières quelconques. Par ailleurs, une combinaison de méthodes diverses permet de traiter le cas global.

\vspace{0.2cm}
\noindent\textbf{Remerciements.} Je remercie mon directeur de stage, Jean-Louis Colliot-Thélène, qui~m'a~proposé ce problème, pour son enthousiasme, ses conseils et suggestions extrêmement précieux dont la suggestion d'appliquer la méthode des fibrations pour les zéro-cycles dans le~cas irrégulier.
\medskip

\subsection*{Les conventions}

\begin{itemize}
    \item $k$ est toujours un corps parfait, on suppose $\car(k)\neq 2$;
    \item $\bar k$ est une clôture algébrique de $k$;
    \item pour $k$ un corps de nombres on note $k_v$ la \textit{complété} de $k$ à la place $v$;
    \item l'expression "\textit{pour presque toute place}" toujours signifie "pour toute place en dehors d'un ensemble fini"
    \item $\bigoplus_v$, $\prod_v$ signifient la somme (resp. le produit) pour chaque place $v$ de $k$;
    \item pour un schéma $X$ on note $\Br(X)=H^2_{\text{ét}}(X,\bbG_m)$ le \textit{groupe de Brauer étale};
    \item une \textit{variété} est un schéma séparé de type fini sur un corps;
    \item pour une variété $X$ sur un corps $k$, étant donné une extension $K/k$, on note $X(K)$ l'ensemble des points de $X$ à~valeurs dans $K$ et $X_K=X\otimes_{k} K$ 
    l'extension des constantes;
    \item un corps $k$ est dite \textit{fertile} si pour chaque variété lisse géométriquement intègre l'hypothèse $X(k)\neq \varnothing$ implique que $X(k)$ est Zariski-dense dans $X$; tout corps $p$-adique est fertile ainsi que le corps des réels;
    \item une \textit{quadrique} $Q\subset \bbP^n$ est une variété définie par l'annulation d'un polynôme homogène de degré 2.

\end{itemize}
\section{Préliminaires} 
\subsection*{Les principes locaux-globaux}
Dans toutes les définitions et conjectures qui suivent $X$ est une variété propre lisse géométriquement intègre sur un corps de nombres $k$. On note $X(\mathbf{A}_k)=\prod_v X(k_v)$ l'ensemble des \textit{points adéliques} de $X$. Comme la variété $X$ est propre, la topologie adélique sur $X(\mathbf{A}_k)$ coïncide avec la topologie de produit. 

\begin{definition}
La variété $X$ satisfait \textit{le principe de Hasse} si 
$X(\mathbf{A}_k)\neq \varnothing$ implique $X(k) \neq \varnothing$. Si, de plus, $X(k)$ est dense dans $X(\mathbf{A}_k)$ par rapport à la topologie de produit, on dit que $X$ \textit{satisfait l'approximation faible}.
Une variété géométriquement intègre $Y$ (possiblement singulière) satisfait le  \textit{principe de Hasse lisse} si tout modèle propre lisse de $Y$ satisfait le principe de Hasse. 
\end{definition}

 Étant donné que la propriété d'avoir un point rationnel est un invariant birationnel pour les variétés propres lisses \cite[Corollary 3.6.16]{PY}, si l'un des modèles propres lisses de $Y$ satisfait le principe de Hasse, alors tous les autres le satisfont.

Le principe de Hasse ne vaut pas en général pour toute variété propre lisse. L'obstruction de Brauer-Manin mesure le défaut du principe de Hasse. On définit un accouplement 
\begin{equation}\label{coupl}
\langle\cdot \,,\cdot \rangle\colon X(\mathbf{A}_k)\times \Br(X)\to \bbQ/\bbZ\end{equation}
comme 
$$\langle x, \alpha \rangle = \sum_v \inv_v\big(\alpha|_{x_v}\big),$$
où $\alpha|_{x_v}\in \Br(k_v)$ est la restriction de $\alpha$ à $x_v\in X(k_v)$ et $\inv_v\colon \Br(k_v)\to \bbQ/\bbZ$ est l'invariant local. On note $X(\mathbf{A}_k)^{\Br}\subseteq X(\mathbf{A}_k)$ le noyau de cet accouplement. C'est un fermé de $X(\mathbf{A}_k)$.

La théorie des corps de classes nous fournit une suite exacte \cite[chap. VII \S9-10]{CF}
\begin{equation}\label{Brsuite}
\xymatrix@=1.3cm{
 0  \ar[r] & \Br(k) \ar[r] & \displaystyle{\bigoplus_v \Br(k_v)} \ar^{\quad\sum \inv_v}[r] & \bbQ/\bbZ \ar[r] & 0. }
 \end{equation}
 Par conséquent, on a une inclusion $X(k)\subseteq X(\mathbf{A_k})^{\Br}$. 
 
 Les deux conjectures suivantes furent formulées par Colliot-Thélène \cite{CT03} dans l'hypothèse que la variété $X$ est rationnellement connexe.

\vspace{2mm}
\noindent\textbf{Conjecture (Br-PH).} \label{(Br-PH)} L'obstruction de Brauer-Manin au principe de Hasse est la seule pour $X$ rationnellement connexe, autrement dit $X(\mathbf{A_k})^{\Br}\neq \varnothing$ implique $X(k)\neq \varnothing$.

\vspace{2mm}
\noindent\textbf{Conjecture (Br-AF).} \label{(Br-AF)} L'obstruction de Brauer-Manin à l'approximation faible est la seule pour $X$ rationnellement connexe, c'est à dire $X(k)$ est dense dans  $X(\mathbf{A_k})^{\Br}$ par rapport à la topologie de produit.\vspace{2mm}
 

Donnons maintenant des rappels sur la situation pour les zéro-cycles.
Les définitions et conjectures suivantes  remontent à l'article \cite{CTSn} et aux généralisations envisagées dans \cite[\S 7]{KS} et \cite{CT95}. Elles firent l'objet d'une série de travaux, en particulier du travail de Salberger
\cite{S88}.

Soit $Z_0^1(X)$ l'ensemble des zéro-cycles de degré 1 sur $X$ et soit $Z_0^1(X)_{\mathbf{A}_k}=\prod_vZ_0^1(X_{k_v})$ l'ensemble des zéro-cycles adéliques de degré 1. L'accouplement \eqref{coupl} induit un accouplement 
\begin{equation}\label{couplcycle}
\langle\cdot \,,\cdot \rangle\colon Z_0^1(X)_{\mathbf{A}_k}\times \Br(X)\to \bbQ/\bbZ. 
\end{equation}
On note $Z_0^1(X)_{\mathbf{A}_k}^{\Br}\subseteq Z_0^1(X)_{\mathbf{A}_k}$ le noyau de \eqref{couplcycle}. De manière analogue, la suite \eqref{Brsuite} assure l'inclusion $Z_0^1(X)\subseteq Z_0^1(X)_{\mathbf{A}_k}^{\Br}$. 


\vspace{2mm}
\noindent\textbf{Conjecture (Br-ZC).} \label{(Br-ZC)} {L'obstruction de Brauer-Manin à l'existence d'un zéro-cycle de degré 1 est la seule pour toute variété $X$ propre lisse et géométriquement intègre, autrement dit $Z_0^1(X)_{\mathbf{A}_k}^{\Br} \neq \varnothing$ implique $Z_0^1(X)\neq \varnothing$. \vspace{2mm}

La conjecture \hyperref[(E)]{(E)} ci-dessous, exprim\'ee sous cette forme dans \cite[(1.4)]{Wn}, 
regroupe,  avec une pr\'ecision de van Hamel, les \'enonc\'es
\cite[Conjectures 1.5]{CT95} portant l'un sur l'existence de z\'ero-cycles de degr\'e 1,
l'autre sur le groupe de Chow des z\'ero-cycles de degr\'e z\'ero. En particulier, cette conjecture généralise \hyperref[(Br-ZC)]{(Br-ZC)}.

On note par $CH_0(X)$ le groupe de Chow des zéro-cycles sur~$X$ modulo l'équivalence rationnelle. Pour une place $v$ de $k$ on définit $$CH_0(X_{k_v})'=
\begin{cases}
CH_0(X_{k_v}) & \text{si $v$ est finie;}\\
CH_0(X_{k_v})/N_{\bar k_v/ k_v}\big(CH_0(X_{\bar k_v})\big) &  \text{si $v$ est archimédienne.}
\end{cases}
$$  
Pour un groupe abélien $M$, on dénote par $\displaystyle{\widehat{M}=\mathop{\lim_{\longleftarrow}}_{n} M/nM}$ la complétion profinie~de~$M$. 
Vu la suite \eqref{Brsuite}, l'accouplement \eqref{coupl} induit un complexe de groupes 
\begin{equation}\label{complexE}
\xymatrix{
 \widehat{CH_0}(X)  \ar[r] & \displaystyle{\prod\limits_v \widehat{CH_0}(X_{k_v})'}  \ar[r] & \Hom(\Br(X),\bbQ/\bbZ). }
 \end{equation}

\noindent\textbf{Conjecture (E).}\label{(E)} 
Le complexe \eqref{complexE} est exact pour toute variété $X$ propre lisse et géométriquement intègre.


\vspace{2mm}
Dans la section 3 on utilisera le théorème de fibration pour les zéro-cycles suivant  \cite[Corollary 8.4]{HW}: 
\begin{theorem}[Harpaz-Wittenberg]\label{prop15}
Soit $k$ un corps de nombres. Soit $X$ une variété propre lisse géométriquement intègre sur $k$ et soit $f\colon X\to\bbP^n_k$ un morphisme dominant. Supposons que la fibre générique de $f$ est rationnellement connexe et que pour tout point fermé $p\in U$ dans un ouvert $U\subset \bbP^n_k$ la fibre $f^{-1}(p)$ satisfait \hyperref[(E)]{$\mathrm{(E)}$} sur le corps $k(p)$. Alors la conjecture \hyperref[(E)]{$\mathrm{(E)}$} vaut pour $X$.
\end{theorem}

 Un théorème de Liang assure la conjecture \hyperref[(E)]{(E)} pour une variété rationnellement connexe si l'obstruction de Brauer-Manin à l'approximation faible est la seule pour toute extension finie \cite[Theorems A,B]{L}.
 
 \begin{theorem}[Liang]\label{TheoremLiang}
     Soit $k$ un corps de nombres et soit $X$ une variété propre lisse géométriquement intègre et rationnellement connexe sur $k$. Supposons que $X_K$ satisfait \linebreak\hyperref[(Br-AF)]{$\mathrm{(Br\text{-}AF)}$} pour chaque extension finie $K/k$. Alors $X_K$ satisfait la conjecture \hyperref[(E)]{$\mathrm{(E)}$} pour chaque extension finie $K/k$.
 \end{theorem}

\subsection*{Les intersections de deux quadriques}

Soit $k$ un corps parfait, $\car(k)\neq 2$.

\begin{definition}
Une \textit{intersection complète de deux quadriques} $X\subset \bbP^n_k$ est un sous-schéma fermé défini par des équations $F=G=0$ où $F,G$ sont des polynômes homogènes de degré 2 sans facteur commun. On appelle \textit{pinceau de} $X$ le pinceau des quadriques défini par les formes quadratiques $\lambda F +\mu G$ où $\lambda,\mu\in k$ ne s'annulent pas simultanément. Le pinceau $\lambda F +\mu G$ avec $\lambda,\mu \in \bar k$ s'appelle le \textit{pinceau géométrique de}~$X$. On dit qu'un pinceau de quadriques est \textit{non dégénéré} s'il contient au moins une quadrique lisse. 
\end{definition}
On note que, si le pinceau de $X$ est non dégénéré, il contient un nombre fini de quadriques singulières. Par conséquent, quand le corps $k$ est infini, le pinceau de $X$ sur $k$ est non dégénéré si et seulement si le pinceau géométrique de $X$ l'est. 

Dans le lemme qui suit on rappelle les résultats standard sur les intersections de deux quadriques formulés dans \cite{CSS}.  
\begin{lemma}\label{lemma0}
Soit $X\subset \bbP^n_k\;(n\geqslant 3)$ une intersection complète géométriquement intègre de deux quadriques $F=0,\,G=0$ et soit $\chi(\lambda,\mu)=\det(\lambda F +\mu G)$ le polynôme caractéristique. Alors
\begin{itemize}
    \item[(i)] Si le polynôme $\chi(\lambda,\mu)$ est identiquement nul, autrement dit, si le pinceau de $X$ est dégénéré, alors $X$ possède un $k$-point singulier. 
    \item[(ii)] Si $X$ possède un $k$-point singulier non conique, $X$ est $k$-birationnelle à une quadrique géométriquement intègre dans $\bbP^{n-1}_k$.
    \item[(iii)] Si pour un point $(\lambda, \mu) \in \bbP^1(\bar k)$ on a $\rk(\lambda F +\mu G) = r$, alors $(\lambda,\mu)$ est une racine du polynôme $P$ d'ordre $\geqslant n+1-r$.
    \item[(iv)] Si le pinceau de $X$ est non dégénéré, $X$ n'est pas conique.
\end{itemize}

De plus, supposons que $k$ est un corps de nombres et que $X$ est une intersection complète géométriquement intègre non conique. Alors
\begin{itemize}
    \item[(v)] S'il existe une forme de rang $\leqslant n-2$ dans le pinceau de $X$, le principe de Hasse lisse vaut pour $X$.
    \item[(vi)] Si $X(k)\neq \varnothing$, le principe de Hasse lisse vaut pour $X$.
\end{itemize}
\end{lemma}
\begin{proof}
\begin{itemize}
    \item[(i)] \cite[Lemma 1.14]{CSS}
    \item[(ii)] \cite[Proposition 2.1]{CSS}
    \item[(iii)] \cite[Lemma 1.15]{CSS}
    \item[(iv)] Comme le pinceau de $X$ est non dégénéré, $X$ est contenue dans une quadrique lisse $Q\subset \bbP^n$. Supposons que $X$ est un cône avec un sommet $P$, alors $X$ est contenue dans l'espace tangent à $Q$ au point $P$ qui est un hyperplan.
    Ceci contredit le fait que $X$ engendre $\bbP^n$ par \cite[Lemma 1.3]{CSS}
    \item[(v)] \cite[Proposition 3.14]{CSS}
    \item[(vi)] Soit $P\in X(k)$ est lisse et l'énoncé est évident, soit $P$ est singulier et on conclut par (ii).
\end{itemize}
\end{proof}

Une combinaison des résultats de \cite{CorTs} et \cite{CSS} assure que l'obstruction de Brauer-Manin à l'approximation faible est la seule pour les intersections singulières de deux quadriques dans $\bbP^4$.

\begin{proposition}\label{prop14}
Soit $k$ un corps de nombres et soit $X\subset \bbP^4_k$ une intersection complète géométriquement intègre non conique de deux quadriques. Supposons que $X$ est singulière et soit $\hat X$ un modèle propre lisse de $X$. Alors $\hat X$ satisfait la conjecture \hyperref[(Br-AF)]{$\mathrm{(Br\text{-}AF)}$}.
\end{proposition}
\begin{proof}
Si $X$ admet une singularité non isolée, l'approximation faible est établie dans \cite[Proposition 3.15]{CSS}. Supposons que les singularités de $X$ sont isolées. Si $X$ est une surface de Châtelet (resp. \textit{une surface de Iskovskih} dans la terminologie de \cite{CorTs}), on conclut par \cite[Theorem 8.11 (a),(c)]{CSS}. Si $X$ n'est pas une surface de Châtelet, alors \cite[Theorem 7.2(b)]{CorTs} assure le principe de Hasse pour $\hat X$. Puisque $\hat X(k)\neq \varnothing$ implique que $\hat X$ est $k$-rationnelle \cite[Theorem 7.2(c)]{CorTs}, $\hat X$ satisfait également l'approximation faible.\end{proof}

Le résultat suivant a été établi par Salberger \cite[Theorem 3.4]{Sb}. On peut trouver une~démonstration pour le cas lisse dans \cite[Théorème 2.20]{CT}, voir aussi \cite[Proposition 5.2.6]{Har}. 

\begin{theorem}[Salberger]\label{TSal}
Soit $k$ un corps de nombres. Soit $X\subset \bbP^n_k\,(n\geqslant 6)$ une~intersection complète géométriquement intègre non conique de deux quadriques. Supposons que $X$ contient une conique $C\subset \bbP^n_k$. Alors le principe de Hasse lisse vaut pour $X$.
\end{theorem}

Si la conique $C$ est singulière, elle contient un sommet rationnel et le principe de Hasse lisse résulte du Lemme \ref{lemma0}(vi).

Une propriété importante des intersections de deux quadriques est l'équivalence entre l'existence d'un point rationnel et d'un zéro-cycle de degré 1 qui résulte du théorème de Amer-Brumer et du théorème de Springer \cite[Proposition 2.4]{CT}.
\begin{theorem}\label{prop16}
Soit $X\subset \bbP^n_k\; (n\geqslant 3)$ une intersection de~deux quadriques sur un~corps~$k$, $\car(k)\neq 2$. Supposons que $X$ admet un point dans une extension $K/k$ de degré impair. Alors $X$ possède un $k$-point. En particulier, si $Z_0^1(X)\neq 0$, on a $X(k)\neq\varnothing$. 
\end{theorem}
\begin{corollary}\label{corcycle-point}
Soit $X\subset \bbP^n_k\; (n\geqslant 3)$ une intersection géométriquement intègre de deux quadriques sur un~corps~$k$, $\car(k)\neq 2$ et soit $\hat X$ un modèle propre lisse de $X$. Alors $Z_0^1(\hat X)\neq 0$ implique $X(k)\neq 0$. 
\end{corollary}
\begin{proof}
Effectivement, $\hat X$ est birationnelle à $X$ et on peut remplacer un~zéro-cycle de degré 1 sur $\hat X$ par un zéro-cycle rationnellement équivalent à support dans l'ouvert de $\hat X$ où l'isomorphisme birationnel est défini. Ainsi, on conclut que $Z_0^1(X)\neq \varnothing$ et donc $X(k)\neq \varnothing$ par le Théorème \ref{prop16}. \end{proof}

Finalement, on est prêt à établir l'équivalence entre les principes locaux-globaux divers pour les intersections de deux quadriques. On rappelle le cas $\mathrm{(E_5)}$ introduit dans \cite[\S6 p.93]{CSS} qui correspond  aux intersections de quadriques dans $\bbP^5$ définies par deux formes quadratiques de rang 4.

\begin{proposition}\label{core}
Soit $X\subset \bbP^n_k\; (n\geqslant 4)$ une intersection complète géométriquement intègre non conique de~deux quadriques sur un corps de nombres $k$ et soit $\hat X$ un modèle propre lisse~de~$X$. Si $n=5$ on suppose qu'on n'est pas géométriquement dans le cas $\mathrm{(E_5)}$. Alors les propriétés suivantes sont équivalentes 
\begin{itemize}
    \item[(i)]  $\hat X_K$ satisfait la conjecture \hyperref[(Br-PH)]{$\mathrm{(Br\text{-}PH)}$} pour chaque extension finie $K/k$;
    \item[(ii)] $\hat X_K$ satisfait la conjecture \hyperref[(Br-AF)]{$\mathrm{(Br\text{-}AF)}$} pour chaque extension finie $K/k$;
    \item[(iii)] $\hat X_K$ satisfait la conjecture \hyperref[(E)]{$\mathrm{(E)}$} pour chaque extension finie $K/ k$;
    \item[(iv)]  $\hat X_K$ satisfait la conjecture \hyperref[(Br-AF)]{$\mathrm{(Br\text{-}ZC)}$} pour chaque extension finie $K/k$.
\end{itemize}
Si $n\geqslant 5$, on a $\Br(\hat X)/\Br(k)=0$ et ces propriétés sont équivalentes au principe de Hasse lisse pour $X_K$ pour chaque extension finie $K/ k$.
\end{proposition}
\begin{proof}
L'implication $\mathrm{(iii)} \Rightarrow \mathrm{(iv)}$ est triviale. L'implication $\mathrm{(ii)} \Rightarrow \mathrm{(iii)}$ résulte du Théorème \ref{TheoremLiang} puisque $\hat X$ est géométriquement rationnelle. 

Montrons que (iv) implique (i). Supposons que $\hat X_K(K_v)^{\Br}$ est non vide pour chaque place $v$ de $K$. Ainsi, $Z_0^1 (\hat X)_{\mathbf{A}_K}^{\Br}\neq 0$ et on conclut que $\hat X$ admet un zéro-cycle de degré 1 sur~$K$. Alors $X(K)\neq \varnothing$ par le Corollaire \ref{corcycle-point} et $X_K$ satisfait le principe de Hasse lisse d'après le Lemme \ref{lemma0}(vi).

Il reste à vérifier que (i) implique (ii). Si $n\geqslant 5$ cela résulte de \cite[Theorem 3.11]{CSS} et de \cite[Théorème 1]{CTSk} comme on n'est pas dans le cas $\mathrm{(E_5)}$. Si $n=4$ et $X$ est lisse on conclut par \cite[Theorem 0.1]{SbSk}. Si $n=4$ et $X$ est singulière, on applique la~Proposition~\ref{prop14}. 

Quand $n\geqslant 5$ \cite[Theorem 3.8]{CSS} joint à \cite[Théorème 4]{CTSk} assure que le groupe de Brauer relatif $\Br(\hat X)/\Br(k)$ s'annule et, donc, (i) est équivalente au principe de Hasse lisse pour $X_K$ pour chaque extension finie $K/ k$.
\end{proof}

\subsection*{Les formes quadratiques}
Pour construire des sections suffisamment génériques d'une intersection de deux quadriques on utilisera le lemme suivant sur le rang de la restriction d'une forme quadratique. 

Soit $k$ un corps parfait, $\car(k)\neq 2$.

\begin{lemma}\label{lemme17}
Soit $Q\subset \bbP^n_k$ une quadrique de rang~$r$. On considère un sous-espace projectif  $V\subset \bbP^n_k$ de dimension $d<n$. Soit $r'$ le rang de l'intersection $Q\cap V$. 

\begin{itemize}
    \item[(i)] Supposons $r-2(n-d)\geqslant 0$. Alors on a $r'\leqslant r-2(n-d)$. L'égalité est possible seulement si $V$ contient le lieu singulier de $Q$.
    \item[(ii)] Supposons $r'<r$. Alors il existe un sous-espace $W\supset V$ de dimension $d+1$ tel que le rang de $Q\cap W$ est au moins $r'+1$.
\end{itemize}
\end{lemma}

\begin{proof}
\begin{itemize}
    \item[(i)] On choisit une chaîne d'espaces projectifs $V=V_0\subset V_1\subset\dots\subset V_{n-d}=\bbP^n_k$ telle que chaque $V_i$ est un hyperplan dans $V_{i+1}$. Ensuite, l'application successive de \cite[Lemma 1.16]{CSS} donne le résultat souhaité. 
    \item[(ii)] D'abord on suppose que les quadriques $Q$ et $Q\cap V$ sont lisses. Alors on a une décomposition orthogonale au niveau d'espaces vectoriels associés $k^{n+1}=\tilde V\oplus \tilde V^{\perp}$ et il suffit de prendre $W=\bbP(\tilde V\oplus \langle x \rangle)$ où $x\in \tilde V^{\perp}$ n'est pas isotrope.

    Maintenant considérons le cas général. Comme $Q\cap V$ est de rang $r'$, il existe un~sous-espace $V'\subset V$ de dimension $r'-1$ tel que  et $Q\cap V'$ est lisse. On peut trouver un sous-espace $V''\supset V'$ de dimension $r-1$ tel que $Q\cap V''$ est lisse puisque $Q$ est de rang~$r$. En appliquant le résultat précédant à $V',V''$ on trouve $x\in V''$ tel que le rang de $Q\cap \langle V', x\rangle$ est au moins $r'+1$. Ainsi, on peut prendre le sous-espace $W=\langle V, x\rangle$ qui satisfait la condition car $W\supset \langle V', x\rangle$. 
\end{itemize}
\end{proof}

On note que si la conclusion du Lemme \ref{lemme17}(ii) vaut, il existe un ensemble ouvert dense de $W\supset V$ de dimension $d+1$ tels que $\rk(Q\cap W)\geqslant r'+1$ puisque la dernière condition est ouverte. 

Un résultat classique sur les hyperplans simultanément tangents à deux quadriques nous sera également utile. 

\begin{lemma}\label{lemmeq}
Soit $Q_A,Q_B\subset \bbP^n_k\;(n\geqslant 2)$ deux quadriques lisses distinctes définies par des matrices symétriques $A,B\in \Mat_{n+1}(k)$. On note $Q'$ la quadrique définie par la matrice $AB^{-1}A$. 
\begin{itemize}
    \item[(i)] Soit $P\in Q_A$ un point rationnel. L'hyperplan $T_{Q_A,P}$ est tangent à $Q_B$ si et seulement si $P\in Q'$.
    \item[(ii)] La sous-variété $V\subset (\bbP^n)^{\vee}$ des hyperplans simultanément tangents à $Q_A,Q_B$ est de codimension 2.
\end{itemize}
\end{lemma}
La quadrique $Q'$ est parfois appelée \textit{la quadrique duale à $Q_B$ par rapport à $Q_A$}.
\begin{proof}
\begin{itemize}

    \item[(i)] Les quadriques $Q_A,Q_B$ sont définies par les équations $X^tAX=0$, $X^tBX=0$. L'espace tangent $T_{Q_A,P}$ est le noyau de la forme linéaire $(AP)^t$. Supposons que $T_{Q_A,P}$ est tangent à $Q_B$ au point $S\in Q_B$. Alors $S^tBS=0$ et $AP=\alpha BS$ avec $\alpha \in k$. On note que $$P^tAB^{-1}AP=(AP)^tB^{-1}(AP)=\alpha^2 (BS)^tB^{-1}(BS)=\alpha^2 S^tBS=0,$$ donc $P\in Q'$. Maintenant supposons que $P\in Q'$. Soit $S=B^{-1}AP$. Alors 
    $$S^tBS=(B^{-1}AP)^tB(B^{-1}AP)=P^tAB^{-1}AP=0,$$
    $$(AP)^tS=(AP)^t(B^{-1}AP)=P^tAB^{-1}AP=0.$$
    Ainsi, $S\in Q_B\cap T_{Q_A,P}$ est un point de tangence puisque $BS=AP$. 
    \item[(ii)] Par (i) $V$ est isomorphe à $Q_A\cap Q'$. Il faut vérifier que cette intersection est complète, c'est à dire que les matrices $A,AB^{-1}A$ ne sont pas proportionnelles. Effectivement, si $A=\beta AB^{-1}A$ avec $\beta\in k$, on déduit l'égalité $B=\beta A$ qui contredit l'hypothèse que les quadriques $Q_A, Q_B$ soient distinctes. 
\end{itemize}
\end{proof}

\subsection*{Les plans hyperboliques}

On dit qu'une forme quadratique $T$ sur un corps $k$ \textit{scinde $m$ plans hyperboliques} si elle peut être décomposée comme 
$$T=\underbrace{\langle -1,1\rangle \perp \langle -1,1\rangle \perp\dots\perp \langle -1,1\rangle}_{m}\perp T'.$$

La propriété de scinder $m$ plans hyperboliques peut être interprétée géométriquement en termes des sous-espaces projectifs contenus dans la quadrique correspondante.

\begin{lemma}\label{lemme4}
Soit $T(x_0,x_1,\dots,x_n)$ une forme quadratique non nulle sur un corps~$k$, $\car(k) \neq 2$. On note $Q\subset \bbP^n_k$ la quadrique définie par $T$ et $Q_{\sing}=\bbP(\ker(T))$ le lieu des points singuliers de $Q$. Alors les énoncés suivants sont équivalents:
\begin{itemize}
\item[(i)] La quadrique $Q$ contient un sous-espace projectif $\Pi\subset \bbP^n_k$ de dimension $m$ tel que $\Pi\cap Q_{\sing}=\varnothing$.
\item[(ii)] La forme $T$ scinde $m+1$ plans hyperboliques.
\end{itemize}
\end{lemma}
\begin{proof} 
On peut se ramener au cas bien connu où la forme $T$ est non dégénérée \cite[Proposition 2.1]{CT}. 

Soit $r=\rk(T)\geqslant 1$. Il existe un sous-espace projectif $W\subset \bbP^n_k$ de dimension $r-1$ tel que $Q\cap W$ est une quadrique lisse. Alors il suffit d'appliquer \cite[Proposition 2.1]{CT} à la quadrique $Q\cap W$ définie par la forme quadratique $F|_W$ en remplaçant $\Pi$ par l'image de $\Pi$ dans $W$ sous la projection $\bbP^n_k\backslash Q_{\sing} \to W$ depuis $Q_{\sing}$.
\end{proof}

 On note par $\Sym_n\simeq \bbA^{n(n+1)/2}$ l'espace des matrices symétriques $n\times n$ vu comme l'espace affine des formes quadratiques en $n$ variables. On dit qu'une forme quadratique en $2n$ variables est \textit{purement hyperbolique} si elle scinde $n$ plans hyperboliques. Le lemme qui suit monte que cette propriété est ouverte dans la topologie $p$-adique.

\begin{lemma}\label{lemmehyperbolique}
Soient $k$ un corps $p$-adique (où réel) et $n\geqslant 1$ un entier. Alors l'ensemble~$U$ des $k$-formes quadratiques purement hyperboliques en $2n$ variables est ouvert dans $\Sym_{2n}(k)$ par rapport à la topologie $p$-adique (resp. réelle).
\end{lemma}
\begin{proof}
Soit $\Sym_{2n}^{\circ}\subset \Sym_{2n}$ l'ouvert des formes non dégénérées. On considère la variété $\mathcal{U}\subset \Sym_{2n}^{\circ}\times \Gr_{n-1}(\bbP^{2n-1}_k)$ des paires $(F,\Pi)$ où $F$ est une forme quadratique non dégénérée de $2n$ variables et $\Pi\subset \bbP^{2n-1}$ est un sous-espace projectif de dimension $n-1$ contenu dans la quadrique $Q_F\subset \bbP^{2n-1}$ définie par $F$. D'après le Lemme \ref{lemme4}, $U$~est~l'image de $\mathcal{U}(k)$ sous la projection $\pi\colon\mathcal{U}\to \Sym_{2n}^{\circ}$. 

Montrons que $\pi$ est un morphisme lisse. On peut passer à la clôture algébrique $\bar k/k$ et démontrer que $\pi_{\bar k}\colon\mathcal{U}_{\bar k}\to (\Sym_{2n}^{\circ})_{\bar k}$ est lisse. La fibre générique géométrique de $\pi_{\bar k}$ est lisse par \cite[Theorem 1.2]{R}. Par conséquent, il existe un point géométrique $F\in \Sym_{2n}^{\circ}(\bar k)$ tel que $\pi_{\bar k}$ est lisse en $F$. On note que le groupe $GL_{2n}\big(\bar k\big)$ agit transitivement sur $\Sym_{2n}^{\circ}(\bar k)$ et sur $\Gr_{n-1}(\bbP^{2n-1}_{\bar k})$. Cette action préserve $\mathcal{U}_{\bar k}$ et est compatible avec la projection $\pi_{\bar k}$. Étant donné cette structure d'espace homogène, on conclut que  $\pi_{\bar k}$ est lisse en tout point géométrique $G\in\Sym_{2n}^{\circ}(\bar k)$ puisqu'il est lisse au point $F$. Comme la lissité est une propriété locale on conclut que $\pi_{\bar k}$ et, donc, $\pi$ est un morphisme lisse.

 On a un morphisme associé de variétés analytiques $\pi\colon \mathcal{U}(k)\to \Sym_{2n}^{\circ}(k)$ qui est une submersion sur son image car $\pi$ est lisse. Par le théorème des fonctions implicites \cite[Theorem p.85]{S} cette submersion admet une section analytique partout localement. En particulier, $\pi\big(\mathcal{U}(k)\big)$ est ouvert dans $\Sym_{2n}^{\circ}(k)$ et, donc, dans $\Sym_{2n}(k)$.
\end{proof}
On déduit du lemme précédent que la propriété de scinder $m$ plans hyperboliques est également ouverte.

\begin{corollary}\label{corouvert}
Soient $k$ un corps $p$-adique (ou réel) et $n\geqslant 1$ un entier. Alors l'ensemble~$U$ des $k$-formes quadratiques en $n$ variables qui scindent $m$ plans hyperboliques est ouvert dans $\Sym_{n}(k)$ par rapport à la topologie $p$-adique (resp. réelle).
\end{corollary}
\begin{proof}
Soit $F\in U$ une $k$-forme quadratique qui scinde $m$ plans hyperboliques. Alors il existe un sous-espace projectif $W\subset \bbP_k^{n-1}$ de dimension $2m-1$ tel que la restriction $F|_W$ est purement hyperbolique. Par le Lemme \ref{lemmehyperbolique}, pour toute $k$-forme $F'$ dans un voisinage de $F$, la restriction $F'|_W$ est purement hyperbolique, donc, $F'$ scinde $m$ plans hyperboliques.
\end{proof}

\section{Le cas régulier}

Dans cette section on démontre le principe de Hasse lisse pour une intersection de deux quadriques $X\subset \bbP^7$ qui est \textit{régulière} dans la terminologie de Wang \cite{Wn}, autrement dit, le rang de toute forme dans le pinceau géométrique de $X$ est au moins 7.

\subsection*{L'existence locale d'une conique}

En suivant la stratégie de la~démonstration pour le cas lisse donnée dans \cite{CT}, on commence par montrer l'existence locale d'une conique contenue dans $X$ pour toute place du corps de nombres dans l'hypothèse que $X$ a des points rationnels lisses partout localement. On démontre cet énoncé dans une situation plus générale que le cas régulier (Proposition~\ref{prop8}). 

Pour une variété projective $X\subset \bbP^n$, on note $X_{\lisse}$ le lieu lisse de $X$ et $T_{X,P}$ l'espace tangent à $X$ en point $P\in X_{\lisse}$ vu comme un sous-espace projectif de $\bbP^n$. Le but des trois lemmes suivants est de trouver un point lisse $P\in X$ sur une intersection de deux quadriques dans $\bbP^7$ tel que l'intersection $X\cap T_{X,P}$ soit un cône sur $Y\subset \bbP^4$ avec $Y$ suffisamment générique.

\begin{lemma}\label{lemmeM}
Soit $k$ un corps algébriquement clos, $\car(k)\neq 2$. Soit $X\subset \bbP^7_k$ une intersection complète intègre non conique de deux quadriques donnée par des équations $F=G=0$ où les formes $F,G$ sont non dégénérées. Supposons que pour tout point $P\in X_{\lisse}(k)$ on a $\rk\big(G|_{T_{X,P}}\big)\leqslant 4$. Alors il existe une forme de rang~$\leqslant 4$ dans le~pinceau de $X$. 
\end{lemma}

\begin{proof}
On note $Q_F,Q_G$ les quadriques lisses définies par $F=0,\, G=0$. Soit $P\in X(k)$ un point lisse. Si $\rk\big( G|_{T_{Q_{F},P}} \big)= 7$, par le Lemme \ref{lemme17}(i) on déduit que $\rk\big( G|_{T_{X,P}} \big)\geqslant 7-2 = 5$ puisque $T_{X,P}$ est un hyperplan dans $T_{Q_F,P}$. Ainsi, pour tout point lisse $P\in X(k)$, on a $\rk\big( G|_{T_{Q_{F},P}} \big)= 6$, autrement dit, $T_{Q_{F},P}$ est tangent à la quadrique $Q_{G}$. 
Soient $A,B$ les matrices correspondantes aux formes quadratiques $F,G$. Par~le~Lemme~\ref{lemmeq}(i), la forme quadratique définie par $AB^{-1}A$ s'annule en tout point $P\in X_{\lisse}(k)$, ainsi elle s'annule sur $X$. Par \cite[Lemma 1.3(ii)]{CSS} on peut écrire $AB^{-1}A=aA+bB$ avec $a,b\in k$. Choisissons une base orthogonale pour $Q_F$, alors $A=\Id$ et on a $$B^{-1}=a+bB,\quad bB^2+aB-1=0.$$ 

Comme la matrice $B$ est annulée par un polynôme non nul au plus quadratique, elle admet un sous-espace propre de dimension $\geqslant 4$ (il suffit d'analyser toute décomposition de Jordan possible). Ainsi, pour quelque $\lambda \in k$, le rang de $B-\lambda \Id$ est au plus 4, et on a trouvé une forme dans le pinceau de rang $\leqslant 4$.  
\end{proof}

\begin{lemma}\label{lemme5}
Soit $k$ un corps algébriquement clos, $\car(k)\neq 2$. Soit $X\subset \bbP^n_k\;(n\geqslant 3)$ une intersection complète intègre non conique de deux quadriques. Fixons un point $M\in \bbP^n(k)$. Alors il existe un ouvert $U\subset X_{\lisse}$ non vide tel que pour tout point $P\in U(k)$ le point $M$ n'appartient pas à l'espace tangent $T_{X,P}$.
\end{lemma}
\begin{proof}
L'ensemble $U$ des points satisfaisant l'exigence du lemme est ouvert et,~donc, il suffit de montrer que $U$ est non vide. Supposons que pour tout point $P\in X_{\lisse}(   k)$ l'espace tangent $T_{X,P}$ contient $M$. Comme $X$ est non conique, il existe une quadrique $Q\supset X$ telle que $M$ n'est pas un sommet de $Q$. Alors les points $P\in Q_{\lisse}(   k)$ tels que $M\in T_{Q,P}$ sont contenus dans un hyperplan $H\subset \bbP^n$ qui est l'espace orthogonal à~$M$ par rapport à la forme quadratique définissant $Q$. En particulier, $X$ est contenu dans~$H$ et ceci contredit le Lemme 1.3 de \cite{CSS} qui assure que $X(   k)$ engendre $\bbP^n$.
\end{proof}

\begin{lemma}\label{lemme6}
Soit $k$ un corps algébriquement clos, $\car(k)\neq 2$. Soit $X\subset \bbP^7_k$ une intersection complète intègre de deux quadriques telle que le pinceau de $X$ est non dégénéré et ne contient pas de formes de rang $\leqslant 5$. Soit $Z\subset \bbP^7_{k}$ le sous-schéma fermé de tous les sommets de toute quadrique singulière dans le pinceau. Alors il existe un ouvert $U\subset X_{\lisse}$ non vide tel que pour tout point $P\in U(k)$
\begin{itemize}
    \item[(i)] L'intersection $Z\cap T_{X,P}$ est un ensemble de points isolés. 
    \item[(ii)] L'intersection $X\cap T_{X,P}$ est un cône simple sur $Y\subset \bbP^4$ où $Y$ est une intersection complète géométriquement intègre et non conique de deux quadriques. De plus, le~pinceau de $Y$ est non dégénéré. 
    \end{itemize}
\end{lemma}

\begin{proof}
\begin{itemize}
    \item[(i)]  Clairement, $Z$ est une union de droites et de points isolés. On choisit un sommet $M_i$ de chaque quadrique singulière dans le pinceau de $X$. En utilisant le Lemme \ref{lemme5} on trouve un ouvert $U\subset X_{\lisse}$ tel que pour tout $P\in U(  k)$ l'espace tangent $T_{X,P}$ ne contient aucun des points $M_i$. Clairement, $U$ satisfait la première condition.
    
    \item[(ii)] Soit $X$ donnée par les équations $F=G=0$ où $\rk(F)=\rk(G)=8$ et soit $P\in U(k)$. Par le Lemme \ref{lemme17}(i) on a  $\rk\big((\lambda F+\mu G)|_{T_{X,P}}\big)\geqslant 3$ pour tout $(\lambda,\mu)\in\bbP^1(k)$. En utilisant le Lemme \ref{lemmeM} on trouve un point $P\in X_{\lisse}(k)$ tel que $\rk\big(G|_{T_{X,P}}\big)\geqslant 5$. Comme cette condition est ouverte, en diminuant $U$ si besoin, on peut supposer que $\rk\big(G|_{T_{X,P}}\big)\geqslant 5$ pour tout $P\in U(k)$ et puis on peut appliquer \cite[Lemma 1.11]{CSS} pour conclure que $X\cap T_{X,P}$ est complète et géométriquement intègre. L'intersection $X\cap T_{X,P}$ est un cône sur $Y\subset \bbP^4$, cela implique que $Y$ est également complète et géométriquement intègre. Le pinceau de $Y\subset \bbP^4$ contient la forme $G$ de rang $5$ qui définit, donc, une quadrique lisse. Alors $Y$ n'est pas conique par le~Lemme~\ref{lemma0}(iv).
\end{itemize}
\end{proof}

Le lemme suivant est un corollaire de \cite[Proposition 1.6]{CT}.

\begin{lemma}\label{lemme7}
    Soit $k$ un corps $p$-adique et soit $X\subset \bbP^3_k$ une intersection complète géométriquement intègre de deux quadriques. Supposons qu'il existe une $k$-forme dégénérée dans le pinceau de $X$. Alors $X$ possède deux points distincts $P_1,P_2$ dans une extension quadratique $K/k$ tels que $P_1\cup P_2$ est défini sur $k$. 
\end{lemma}

\begin{proof}  
Par \cite[Proposition 1.6]{CT} $X$ admet un point $P_1$ dans une extension quadratique $K/k$. Si $P_1$ n'est pas un $k$ point il suffit de prendre pour $P_2$ le point conjugué à $P_1$ par $\Gal(K/k)$. Supposons que $P_1$ est un $k$-point. Si la courbe $X$ est lisse, elle admet un autre $k$-point $P_2$ comme $k$ est fertile. Si $X$ est singulière, elle est birationnelle à une conique par le Lemme \ref{lemma0}(ii) pour laquelle l'énoncé est clair.
\end{proof}

La proposition qui suit généralise \cite[Theorem 2]{HB}, \cite[Théorème 7.3]{CT} pour le cas singulier. 

\begin{proposition}\label{prop8}
Soit $k$ un corps $p$-adique et soit $X\subset \bbP^7_k$ une intersection complète géométriquement intègre de deux quadriques telle que le pinceau géométrique de $X$ est non dégénéré et ne contient pas de formes de rang $\leqslant 5$. Supposons que $X$ possède un $k$-point lisse. Alors le pinceau de $X$ contient une $k$-forme quadratique qui scinde 3 plans hyperboliques.
\end{proposition}
\begin{proof}
Par le Lemme \ref{lemme6} et le fait que les points lisses sont denses dans $X$, on peut trouver un point lisse $P\in X(k)$ tel que l'intersection $X\cap T_{X,P}$ est un cône simple sur $Y\subset \bbP^4_k$ où $Y$ est une intersection complète géométriquement intègre non conique de deux quadriques et le pinceau de $Y$ est non dégénéré. De~plus, l'intersection $Z\cap T_{X,P}$ est une union d'un nombre fini de points $S_i\in T_{X,P}$. Soit $\tilde S_i\in \bbP^4$ les images de ces points sous la projection depuis le point $P$.

Supposons que $Y\subset \bbP^4_k$ est donnée par les équations $F=G=0$. Soit $U\subset \bbP^1_{\lambda,\,\mu}$ l'ensemble ouvert non vide des valeurs $(\lambda,\mu)$ telles que la forme $\lambda F +\mu G$ est non dégénérée. Soit $\mathcal{H}\subset (\bbP^4)^{\vee}\times U$ la variété de paires $\big(H,\, [\lambda,\mu]\big)$ telles que l'hyperplan $H$ est tangent à la quadrique lisse $\lambda F+\mu G=0$. Une fibre géométrique de la projection $\varphi\colon \mathcal{H}\to U$ est l'ensemble des hyperplans tangents à une quadrique lisse dans $\bbP^4$, c'est une quadrique lisse dans l'espace dual $(\bbP^4)^{\vee}$. Ainsi, toute fibre géométrique de $\varphi$ est~lisse de~dimension~3, ceci implique que $\mathcal{H}$ est une variété lisse de dimension 4. Si $(\lambda,\mu)\in U(k)$, la fibre $\varphi^{-1}(\lambda,\mu)$ est une quadrique lisse dans $(\bbP^4_{k})^{\vee}$ qui admet un $k$-point puisque toute forme quadratique non dégénérée de rang $5$ est isotrope. Ainsi, $\mathcal{H}(k)$ est non vide et, donc, Zariski-dense dans $\mathcal{H}$ comme $k$ est fertile.

Comme le pinceau de $Y$ est non dégénéré, on peut choisir deux quadriques lisses distinctes $Q_1,Q_2$ dans le pinceau de $Y$. Soit $V\subset (\bbP^4)^{\vee}$ la sous-variété fermée des hyperplans qui sont simultanément tangents aux quadriques $Q_1,Q_2$. Par le Lemme \ref{lemmeq}(ii), $V$ est de codimension 2. Considérons l'autre projection $\pi \colon \mathcal{H}\to (\bbP^4)^{\vee}$. Toute fibre de $\pi$ est contenue dans $U$ et est donc de dimension $\leqslant 1$. En particulier, $\dim \pi^{-1}(V) \leqslant 2 + 1 = 3$ et on a $\dim \pi^{-1}\big((\bbP^4)^{\vee}\backslash V\big)=4$ car $\dim \mathcal{H} =4$.  Pour tout $H\in (\bbP^4)^{\vee}\backslash V$ le polynôme caractéristique $\chi_H(\lambda,\mu)=\det \big((\lambda F+\mu G)|_H\big)$ est non nul et,~donc, la fibre $\pi^{-1}(H)$ est finie. Comme $\dim \pi^{-1}\big((\bbP^4)^{\vee}\backslash V\big) = \dim (\bbP^4)^{\vee} =4 $, ceci garantit que le morphisme $\pi$ est dominant.   

D'après le théorème de Bertini, il existe un ensemble ouvert d'hyperplans $W \subset (\bbP^4)^{\vee}$ tel que pour chaque hyperplan $H\in W$ l'intersection $Y\cap H$ est géométriquement intègre. En diminuant $W$ si nécessaire, on peut supposer qu'aucun hyperplan $H\in W$ ne contient l'un des points $\tilde S_i$. L'image réciproque $\pi^{-1}(W)$ est non vide puisque $\pi$ est dominant. Alors, il existe un $k$-point $\big(H, [\lambda,\mu]\big)\in \pi^{-1}(W)(k)$, autrement dit, l'hyperplan $H$ est défini sur $k$ et tangent à la quadrique lisse $\lambda F+\mu G=0$ où $\lambda,\mu \in k$. On considère l'intersection géométriquement intègre $Y\cap H$. D'après le Lemme \ref{lemme7}, elle admet des points distincts $P_1,P_2$ tels que $P_1\cup P_2$ est défini sur $k$. Soit $\Pi\subset T_{X,P}$ le $k$-plan engendré par les deux droites qui correspondent à ces points. Il existe une quadrique $Q$ dans le pinceau de $X$ qui contient $\Pi$. Par la construction, aucun sommet d'une quadrique singulière dans le pinceau géométrique de $X$ n'est contenu dans $\Pi$ et on conclut par le~Lemme~\ref{lemme4}. 
\end{proof}

\subsection*{Les places réelles}
La proposition suivante est un analogue de la Proposition \ref{prop8} pour les places réelles, voir aussi \cite[Proposition 2.23]{CT} et \cite[Lemma 12.1]{HB}.

\begin{proposition}\label{propr}
Soit $X\subset \bbP^7_{\bbR}$ une intersection complète géométriquement intègre de deux quadriques telle que le pinceau géométrique de $X$ est non dégénéré et ne contient pas de formes de rang $\leqslant 5$. Alors le pinceau de $X$ contient une forme quadratique qui scinde $3$ plans hyperboliques.
\end{proposition}
\begin{proof} 
Soit $X$ donnée par des équations $F=G=0$.
On utilise l'astuce de~Mordell comme dans 
\cite[Proof of Thm. 10.1, (e) p. 114]{CSS}. Ainsi, on fait varier $(\lambda,\mu)$ dans le cercle $S^1_{\lambda,\mu}: \lambda^2+\mu^2=1$. La signature $\sgn(\lambda F+\mu G)$ est une fonction impaire $S^1_{\lambda,\mu}\to \bbZ$ car $\sgn(-\lambda F-\mu G)=-\sgn(\lambda F+\mu G)$. Soit $C\subset S^1_{\lambda,\mu}$ l'ensemble fini des points $(\lambda,\mu)$ tels que $\rk(\lambda F+\mu G)<8$. Alors la signature est une fonction localement constante sur $S^1\backslash C$ et le passage par chaque point $(\lambda,\mu)\in C$ change la signature au plus par $2\cdot(8-6)=4$. Par conséquent, pour quelque $(\lambda,\mu)\in S^1_{\lambda,\mu}\backslash C$ on a $|\sgn(\lambda F+\mu G)|\leqslant 2$ et, donc, la forme $\lambda F+\mu G$ scinde 3 plans hyperboliques.  
\end{proof}

\subsection*{La fin de la démonstration}

Maintenant on est prêt à établir le principe de Hasse dans le cas régulier.

\begin{proposition}\label{casreg}
Soit $X\subset \bbP^7_k$ une intersection complète géométriquement intègre de deux quadriques sur un corps de nombres $k$. Supposons que le pinceau géométrique de $X$ est non dégénéré et ne contient pas de formes de rang $\leqslant 6$. Alors le principe de Hasse lisse vaut pour $X$.
\end{proposition}
\begin{proof}
La variété $F$ paramétrant les plans projectifs qui sont contenus dans le lieu lisse de $X$ est une jacobienne généralisée et est, donc,  géométriquement intègre \cite[Theorem 3.21]{Wn}. D'après les estimations de Lang-Weil \cite{LW}, $F$ admet un $k_v$-point rationnel pour presque toute $v$. Ainsi, $X$ contient un plan projectif sur $k_v$ pour toute place $v$ de $k$ en dehors d'un ensemble fini de places qu'on note~$S$, voir aussi la Proposition \ref{propplan} de l'appendice. Ainsi, pour une place $w\notin S$ toute forme non dégénérée dans le pinceau de $X$ scinde 3 plans hyperboliques par le Lemme \ref{lemme4}. 

D'après la Proposition \ref{prop8} et la Proposition \ref{propr}, pour chaque place $v\in S$ il existe une $k_v$-forme quadratique dans le pinceau de $X_{k_v}$ qui scinde 3 plans hyperboliques. En utilisant le Corollaire \ref{corouvert} et l'approximation faible pour $S$, on trouve une $k$-forme lisse $F$ dans le pinceau de $X$ qui scinde 3 plans hyperboliques sur $k_v$ pour chaque place $v$ de $k$. Alors $F$ scinde 3 plans hyperboliques sur $k$ par un théorème de Hasse \cite[Proposition 2.17]{CT} et la quadrique $Q_F$ définie par $F$ contient un plan projectif $\Pi\subset \bbP^7_k$ d'après le~Lemme~\ref{lemme4}. L'intersection $\Pi\cap X$ est une conique et on conclut par le Théorème \ref{TSal}.
\end{proof}

\section{Le cas irrégulier}

Maintenant on passe au cas d'une intersection de deux quadriques $X\subset \bbP^7$ qui n'est pas régulière. On considère séparément trois cas: le pinceau géométrique de $X$ contient une forme de rang $\leqslant 5$; le pinceau géométrique de $X$ contient 2 formes conjuguées de rang~6; le pinceau géométrique de $X$ contient 4 formes de rang 6. Ensuite, on verra que ces trois cas épuisent toutes les possibilités non couvertes
par des articles antérieurs, en particulier \cite{CSS} (cf. la preuve du Théorème \ref{theoremeprincipal}).

\subsection*{Une forme de rang $\leqslant 5$}

On commence par un lemme arithmétique. 
 \begin{lemma}\label{lemma}
Soient $k$ un corps de nombres et $C$ une conique lisse sur une extension quadratique $K/k$. Alors il existe une extension quadratique $L\supset k$ telle que $C(KL)\neq \varnothing$.
\end{lemma}
\begin{proof}
Soit $\alpha \in \Br(K)[2]$ la classe de la conique $C$. On a une injection
$\Br(K)\into \bigoplus_{w} \Br(K_w)$ qui fait partie de la suite \eqref{Brsuite}. Soit $S$ l'ensemble des places $w$ de $K$ telles que $C(K_w)=\varnothing$, autrement dit, telles que la restriction $\alpha_w\in \Br(K_w)$ est non nulle. Considérons une place $w\in S$, on note par $v$ la trace de $w$ sur $k$. Si $w\in S$ est réelle, on a forcément $K_w=k_v\simeq \bbR$, sinon $K_w$ est algébriquement clos et $\alpha_w=0$. En tout cas, il existe une extension quadratique $L_v \supset k_v$ linéairement disjointe de $K_w\supset k_v$ pour toute $w$ au dessus de $v$ (il y en a au plus 2).  Effectivement, pour une place réelle on peut prendre l'unique extension quadratique $\bbR\subset\bbC$; pour une place non archimédienne, l'ordre de $k_v^*/(k_v^*)^2$ est au moins 4. 

L'application $k^*/(k^*)^2\to \prod_{v\in S} k_v^*/(k_v^*)^2$ est surjective par l'approximation faible. Alors on peut choisir une extension quadratique $k\subset L$ qui donne l'extension $k_v\subset L_v$ pour chaque $v\in S$. Alors aucune $w\in S$ n'est scindée dans $KL$ et on a une extension quadratique induite $K_w \subset (KL)_w=K_wL_v$ pour toute $w$. La~restriction $\Br(K_w)\to \Br\big((KL)_w\big)$ est la multiplication par le degré $\deg[(KL)_w:K_w]=2$. Ainsi, $\alpha_w|_{(KL)_w}=0 \in \Br\big((KL)_w\big)$ puisque $\alpha_w\in \Br(K_w)[2]$. Ceci implique que $\alpha|_{KL}=0$ et $C(KL)\neq \varnothing$.\end{proof}
\begin{proposition}\label{cas5}
Soit $X\subset \bbP^7_k$ une intersection complète géométriquement intègre non~conique de deux quadriques sur un corps de nombres $k$. Supposons qu'il existe une forme quadratique de rang $\leqslant 5$ dans le pinceau géométrique de $X$. Alors le principe de Hasse lisse vaut pour~$X$.
\end{proposition} 
Soit $X$ définie par des équations $F=G=0$. Si le polynôme caractéristique $P(\lambda,\mu)=\det(\lambda F+\mu G)$ est identiquement nul, $X$ est birationnelle à une quadrique géométriquement intègre dans $\bbP^6_k$ par le Lemme \ref{lemma0}(i),(ii).
Soit $(\lambda,\mu)\in\bbP^1(\bar k)$ tel que $\rk(\lambda F +\mu G) \leqslant 5$.  D'après le Lemme \ref{lemma0}(iii) le point $(\lambda,\mu)$ est une racine du polynôme $P$ d'ordre $\geqslant 3$. Si $(\lambda,\mu)\in\bbP^1(k)$, on conclut par le Lemme \ref{lemma0}(v). Sinon, $(\lambda,\mu)\in\bbP^1(K)$ où $K/k$ est une extension quadratique. Alors $P$ a deux racines d'ordre~$\geqslant 3$ qui sont conjuguées par $\Gal(K/k)$. Il suffit de démontrer l'affirmation suivante:

\begin{proposition}
Soit $X\subset \bbP^n_k,\,(n\geqslant 5)$ une intersection complète géométriquement intègre non~conique de deux quadriques sur un corps de nombres $k$. Supposons que $X$ est donnée par deux formes quadratiques $Q_1,\,Q_2$ sur une extension quadratique $K/k$ telles que $Q_1,\,Q_2$ sont conjuguées par $\Gal(K/k)$ et $\rk(Q_1)=\rk(Q_2)\leqslant n-2$. Alors le~principe de Hasse lisse vaut pour~$X$.
\end{proposition}
\begin{proof}
On note $X_i\subset \bbP^n_K$ la quadrique définie par $Q_i$. Soit $m=n-\rk(Q_1)$ et soit $\Pi_i=\bbP\big(\ker(Q_i)\big)\subset \bbP^n_K$ le sous-espace projectif des points singuliers de $X_i$. On a $\dim(\Pi_i)=m\geqslant 2$. Si l'intersection $\Pi_1\cap \Pi_2$ est non vide, $X$ est conique. Donc $\Pi_1\cap \Pi_2=\varnothing$. 
        
On considère les intersections $\Pi_1\cap X_2=Y_1,\, \Pi_2\cap X_1=Y_2$ dans $\bbP^n_{K}$. Étant donnés des points géométriques $P_i\in Y_i$ la droite $\langle P_1, P_2\rangle$ est contenue dans $X$ car $\langle \Pi_1, P_2 \rangle \subset X_1$, $\langle \Pi_2, P_1\rangle \subset X_2$. Si la quadrique $Y_1\subset\Pi_1$ est singulière, elle contient un $K$-point $P$. Alors il y a le point conjugué $P^{\sigma}\in Y_2$ où $\sigma$ est le générateur de $\Gal(K/k)$. La droite $\langle P, P^{\sigma}\rangle$ est donc contenue dans $X$ et définie sur $k$. En particulier, $X$ possède un $k$-point et on peut utiliser le Lemme \ref{lemma0}(vi). 
        
On peut donc supposer que les quadriques $Y_i\subset \Pi_i\simeq \bbP^m_K$ sont lisses et sans \linebreak $K$-points. Dans ce cas, il existe une $K$-conique $C_1\subset Y_1$, notons $C_2=C_1^\sigma\subset Y_2$ la conique conjuguée dans $Y_2$. Le Lemme \ref{lemma} assure qu'il existe une extension quadratique $L\supset k$ telle que $C_1(F)\neq \varnothing$ pour $F=KL$. Soit $\tau$ le générateur de $\Gal(L/k)$. On a un isomorphisme $\Gal(F/k)=\Gal(K/k)\times \Gal(L/k)$. On choisit une paire de $F$-points conjugués $P,\,P^{\tau}$ sur $C_1$. On a également des points $P^{\sigma},\,P^{\tau\sigma}$
sur $C_2$. Considérons les droites $\ell_1=\langle P,P^\sigma\rangle,\, \ell_2=\langle P^\tau,P^{\tau\sigma}\rangle$ contenues dans $X$. On peut supposer que ces droites sont gauches, sinon elles se croisent en un $k$-point et on applique le Lemme~\ref{lemma0}(vi). Puisque les droites $\ell_1,\,\ell_2$ sont définies sur $L$ et conjuguées par $\tau$ on peut conclure par le~Théorème~13.2 de \cite{CSS} dans le cas où $n\geqslant 6$. Si $n=5$, il faut vérifier que le cas exceptionnel $\mathrm{(E_5)}$ n'est pas possible. Effectivement, s'il y a deux formes non proportionnelles de rang 4 dans le pinceau $\lambda F+\mu G$, alors le polynôme $P$ possède deux autres racines d'ordre $\geqslant 2$ dans $\bar k$ par le Lemme \ref{lemma0}(iii) et on obtient une contradiction avec le~fait que $\deg(P)\leqslant 8$. \end{proof}

\subsection*{Deux formes conjuguées de rang~6}

On considère le cas où le pinceau géométrique de $X$ contient deux formes quadratiques conjuguées de rang 6.
Étant donné une telle $X$, on applique la méthode des fibrations pour les zéro-cycles afin de  réduire le problème aux intersections singulières de deux quadriques dans $\bbP^4$ pour lesquelles le résultat est connu.

Dans le lemme technique qui suit on appelle \textit{quadrilatère gauche} un quadrilatère dans $\bbP^3$ formé par quatre sommets distincts qui ne sont pas contenus dans un plan. 
\begin{lemma}\label{lemme18}
Soit $k$ un corps algébriquement clos de caractéristique 0. Soit $X\subset \bbP^7_k$ une intersection complète intègre de deux quadriques donnée par des équations $F=G=0$ où $\rk(F)=\rk(G)=6$ et le pinceau de $X$ est non dégénéré. Supposons que les sommets des quadriques $F=0,\, G=0$ engendrent un sous-espace projectif $\Pi \subset \bbP^7_k$ de dimension 3 et que $X\cap \Pi$ est un  quadrilatère gauche. Soit $H\simeq\bbP^3_k$ la variété qui paramètre les sous-espaces $\Gamma$ de dimension~4 contenant $\Pi$. Alors il existe un ouvert $U\subset H$ tel que pour tout $\Gamma\in U(k)$ l'intersection $X\cap \Gamma$ est complète intègre non~conique. 
\end{lemma}

\begin{proof}
Le pinceau de $X\cap \Pi$ contient au moins deux formes de rang $\leqslant 2$ qui sont les restrictions de $F,G$. Comme le quadrilatère $X\cap \Pi$ n'est pas contenu dans un plan, le rang de toute forme dans le pinceau de $X\cap \Pi$ est au moins 2. Il est facile de voir que le quadrilatère n'est pas contenu dans une quadrique de rang 3 et qu'il est contenu dans un nombre fini de quadriques de rang 2. On conclut que le pinceau de $X\cap \Pi$ est non dégénéré et contient précisément deux quadriques singulières de rang 2 qui sont les restrictions de $F,G$.

Choisissons une forme quadratique $\Psi$ non dégénérée dans le pinceau de $X$. Par~le~Lemme~\ref{lemme17}(ii), on peut trouver un ouvert $U\subset H$ tel que pour tout $\Gamma\in U(k)$ on a $\rk(F|_{\Gamma})= 3$, $\rk(G|_{\Gamma})= 3$, $\rk(\Psi|_{\Gamma})= 5$. Effectivement, ces trois conditions sont ouvertes et le Lemme \ref{lemme17}(ii) assure que les lieux correspondants sont non vides. Soit $\Gamma\in U(k)$. D'après \cite[Lemma 1.11]{CSS} l'intersection $X\cap \Gamma$ est complète et intègre. Elle est non conique car le pinceau de $X\cap \Gamma$ contient la forme $\Psi|_{\Gamma}$ non dégénérée, voir le Lemme \ref{lemma0}(iv). 
\end{proof}

\begin{proposition}\label{cas2conj}
Soit $k$ un corps de nombres. Soit $X\subset \bbP^7_k$ une intersection complète géométriquement intègre de deux quadriques. Supposons que le~pinceau géométrique de $X$ est non dégénéré et contient deux formes $F,G$ de rang 6 distinctes dans une extension quadratique $K/k$ qui sont conjuguées par $\Gal(K/k)$. Alors le principe de Hasse lisse vaut pour $X$. 
\end{proposition}

\begin{proof}
On~note $L_F, L_G$ les droites singulières des quadriques $F=0, G=0$. L'intersection $L_F\cap L_G$ est vide puisque $X$ est non conique par le Lemme \ref{lemma0}(iv). Ainsi, $L_F, L_G$ engendrent un sous-espace $\Pi\subset \bbP^7_k$ de dimension~3. Comme dans la preuve de la Proposition \ref{cas5} on note que pour tous points $P\in X\cap L_F$, $Q\in X\cap L_G$ la droite $PQ$ est contenue dans $X$. Donc, si $X$ contient les droites $L_F,L_G$, alors $X$ contient l'espace $\Pi\simeq \bbP^3_k$ qu'elles engendrent et on applique le Lemme~\ref{lemma0}(vi). Sinon, soit $X\cap L_F=\{P_1,P_2\}, X\cap L_G=\{S_1,S_2\}$ où les points d'intersection peuvent coïncider. S'ils coïncident on a $P_1=P_2$, $S_1=S_2$ et $X$ contient la $k$-droite $P_1S_1$, donc, on peut encore conclure par le Lemme~\ref{lemma0}(vi). Maintenant on suppose que les quatre points $P_1,S_1,P_2,S_2$ sont distincts. Alors $X\cap\Pi$ est le quadrilatère gauche formé par ces points. 

Soit $H\simeq\bbP^3_k$ la variété qui paramètre les sous-espaces $\Gamma$ de dimension~4 contenant $\Pi$. On a une projection $\bbP^7_k\backslash \Pi\to H$. Soit $\pi\colon X\backslash \Pi \to H$ la restriction de cette projection. Étant donné $\Gamma\in H(\bar k)$, la~fibre $\pi^{-1}(\Gamma)$ s'identifie avec l'intersection $(X\backslash \Pi)\cap \Gamma$. 
En éclatant $X$ et en résolvant les singularités (Hironaka), on trouve un modèle propre lisse $\hat X\to X$ tel qu'il existe un prolongement $\hat \pi\colon \hat X\to H$ de l'application rationnelle $\pi$. Clairement, il existe un ouvert $U\subset H$ tel que pour tout $\Gamma\in U(\bar k)$ la fibre $\hat \pi^{-1}(\Gamma)$ est un modèle propre lisse de la fibre $\pi^{-1}(\Gamma)$ qui contient 4 points singuliers $P_1,P_2,S_1,S_2$. Par le Lemme \ref{lemme18}, on peut diminuer $U$ de manière que pour tout $\Gamma\in U(\bar k)$ la~fibre $\hat\pi^{-1}(\Gamma)$ soit un modèle propre lisse d'une intersection complète géométriquement intègre non conique de deux quadriques dans $\bbP^4$. 

Toute fibre géométrique au-dessus de $U$ est de dimension 2. Ceci assure que $\hat \pi$ est dominant. Par la Proposition~\ref{prop14}, la fibre $\hat\pi^{-1}(\Gamma)$ satisfait la conjecture \hyperref[(Br-AF)]{$\mathrm{(Br\text{-}AF)}$} pour tout $\Gamma\in U(k)$. D'après la Proposition \ref{core}, la conjecture~\hyperref[(E)]{(E)} est également satisfaite pour $\hat\pi^{-1}(\Gamma)$. La~fibre générique de $\hat \pi$ est un modèle propre lisse d'une intersection géométriquement intègre non conique de deux quadriques qui est géométriquement rationnelle, donc, rationnellement connexe. Ainsi, on peut appliquer le Théorème~\ref{prop15} pour conclure que la conjecture \hyperref[(E)]{(E)} vaut pour $\hat X$. Donc, $X$ satisfait le principe de Hasse lisse d'après la Proposition \ref{core}.
\end{proof}

\subsection*{Quatre formes de rang 6}

Quand le pinceau géométrique de $X$ contient 4 formes de rang 6, on construit un fibré de $X$ en espaces principaux homogènes de tores algébriques afin d'utiliser la méthode des fibrations pour les zéro-cycles comme dans le cas précédent.  
\begin{proposition}\label{propf}
Soit $k$ un corps algébriquement clos, $\car(k)\neq 2$. Soit $X\subset \bbP^7_k$ une intersection complète intègre de deux quadriques. Supposons que le pinceau de $X$ est non dégénéré et contient $4$ formes quadratiques $\{\Phi_i\}_{0<i\leqslant 4}$ de rang 6. Soit $\bbP^7_k=\bbP(V)$ où $V$ est un espace vectoriel de dimension 8. Alors 
\begin{itemize}
    \item[(i)] L'espace $V$ se décompose en somme directe $V=\bigoplus_i \ker(\Phi_i)$ et cette décomposition est orthogonale par rapport à toute forme dans le pinceau de $X$.
    \item[(ii)] L'intersection $X\cap \bbP(\ker(\Phi_i))$ est une paire de deux points. 
\end{itemize}
\end{proposition}
\begin{proof}
Soit $X$ donnée par des équations $F=G=0$ où $\rk(F)=\rk(G)=8$ et soient $A,B$ les matrices des formes quadratiques $F,G$ respectivement. En choisissant une base orthogonale pour~$G$, on peut supposer que $B=\Id$. Soit $\Phi_i=A-\lambda_i B$ où $\lambda_i\in k$ sont distincts et non nuls. On note $V_i=\ker(\Phi_i)$. Alors les sous-espaces $V_i$ sont des espaces propres de la matrice $A$ avec les valeurs propres $\lambda_i$. Comme chaque $V_i$ est de dimension 2, la matrice $A$ est diagonalisable. On conclut que $V=\bigoplus_iV_i$ et $A|_{V_i}=\lambda_i\Id_{V_i}$. Pour (i) il reste à vérifier que cette décomposition est orthogonale par rapport à toute forme dans le pinceau. 

On note par $\langle\cdot \,,\cdot \rangle_F,\; \langle\cdot \,,\cdot \rangle_G$ les formes bilinéaires correspondantes à $F, G$. Comme la matrice $A$ est symétrique, elle définit un opérateur linéaire autoajoint par rapport à $G$. Donc, pour $x\in V_i, y\in V_j$ on a 
$$\langle x, y\rangle_G=\frac{1}{\lambda_i}\langle Ax,y\rangle_G=\frac{1}{\lambda_i}\langle x,Ay\rangle_G=\frac{\lambda_j}{\lambda_i}\langle x,y\rangle_G.$$
Ceci assure que $V_i$ est orthogonale à $V_j$  par rapport à $G$ pour $i\neq j$. Comme 
$$\langle x,y\rangle_F=x^tAy=\lambda_j x^ty=\lambda_j\langle x, y\rangle_G$$
on déduit que les $V_i$ sont aussi orthogonaux par rapport à $F$ et donc, par rapport à toute forme dans le pinceau de $X$. Pour démontrer (ii) il suffit de vérifier que la restriction de la forme $F$ à $V_i$ est non dégénérée. Effectivement, soit $x\in \ker(F|_{V_i})$, alors $x\in \ker(F)$ puisque $V_i$ est orthogonale à $V_j\; (j\neq i)$. Mais la forme $F$ est non dégénérée, donc $x=0$.
\end{proof}

\begin{proposition}\label{cas4formes}
Soit $k$ un corps de nombres. Soit $X\subset \bbP^7_k$ une intersection complète géométriquement intègre de deux quadriques. Supposons que le pinceau géométrique de $X$ est non dégénéré et contient $4$ formes quadratiques de rang 6. Alors le principe de Hasse lisse vaut pour $X$. 
\end{proposition}

\begin{proof}
Soient $\{\Phi_i\}_{0<i\leqslant 4}$ les formes de rang 4 dans le pinceau de $X$ et soit $K/k$ une extension finie de Galois sur laquelle elles sont définies. On note par $L_i=\bbP(\ker(\Phi_i))$ la droite des sommets de la quadrique $\Phi_i=0$. On choisit deux $k$-formes non dégénérées $F,G$ dans le pinceau de~$X$. Soit $\varphi_i=F|_{L_i}$. Alors on a $$F=\varphi_1+\varphi_2+\varphi_3+\varphi_4$$ où les formes $\varphi_i$ sont de rang 2 par la Proposition \ref{propf}. Le groupe de Galois $\Gal(K/k)$ agit sur les $\varphi_i$ par permutation puisque la forme $F$ est définie sur $k$. Comme la restriction de toute forme dans le pinceau à $\ker(\Phi_i)$ est proportionnelle à $\varphi_i$ on a $G|_{L_i}=\alpha_i\varphi_i$ où les $\alpha_i\in K^*$ sont tous distincts. Donc $X_K$ est définie par les équations 
\begin{equation*}
   X_K\colon \begin{cases}
       \varphi_1+\varphi_2+\varphi_3+\varphi_4=0\\
       \alpha_1\varphi_1+
       \alpha_2\varphi_2+\alpha_3\varphi_3+\alpha_4\varphi_4=0
   \end{cases} 
\end{equation*}

On considère l'espace vectoriel engendré par les $\varphi_i$

$$V=\langle\varphi_1,\varphi_2,\varphi_3,\varphi_4\rangle \subset H^0\big(X, \mathcal{O}(2)\big)\otimes_k K.$$
L'espace $V$ est de dimension 2 et on peut le descendre à $k$ car l'ensemble des générateurs $\varphi_i$ est Galois-stable \cite[Chap. 5 \S10.4]{Br}. Ainsi il existe un sous-espace $V_k \subset H^0\big(X, \mathcal{O}(2)\big)$ de dimension 2 tel que $V_k\otimes_k K= V$. Soit $U\subset X$ l'ouvert de $X$ où aucune des formes $\varphi_i$ n'est nulle. Alors on a un $k$-morphisme dominant $\pi\colon U\to \bbP^1_k$ associé au système linéaire de $V_k$.
Par construction le morphisme $\pi_K\colon U_K\to \bbP^1_K$ est de la forme 
$$\pi_K(x)=[\varphi_1(x):\varphi_2(x):\varphi_3(x):\varphi_4(x)]\in \bbP(V)\subset \bbP\Big(H^0\big(X,\mathcal{O}(2)\big)\Big).$$
Ainsi, pour un point schématique $p\in \pi_K(U_K)$ la fibre $\pi_K^{-1}(p)\subset \bbP^7_K$ est définie par les équations
$$\beta_1\varphi_1=\beta_2\varphi_2=\beta_3\varphi_3=\beta_4\varphi_4\neq 0$$
où $\beta_i\in {k(p)}^{*}$.
Après un changement de base linéaire sur $\overbar{k(p)}$ on peut la réécrire comme
$$x_0x_1=x_2x_3=x_4x_5=x_6x_7\neq 0$$
où $x_i$ sont des coordonnées homogènes dans  $\bbP^7_K$.

Donc, pour tout point $p\in\pi(U)$ la fibre $\pi^{-1}(p)$ est géométriquement isomorphe au tore  $(\mathbb{G}_m)^4$. D'après \cite[p.131 Lemme]{CTSk}, ceci implique que la fibre $\pi^{-1}(p)$ est isomorphe à un espace homogène principal sous un tore sur $k(p)$. Soit $\hat X\to X$ un modèle propre lisse de $X$ tel qu'il existe un prolongement $\hat \pi\colon \hat X\to\bbP^1_k$ du morphisme~$\pi$. Alors il existe un ouvert $W\subset \bbP^1_k$ tel que pour tout point $p\in W$ la fibre $\hat \pi^{-1}(p)$ est un modèle propre lisse d'un espace homogène principal sous un $k(p)$-tore. Ainsi, la fibre générique de $\hat \pi$ est géométriquement rationnelle. Toute fibre $\hat \pi^{-1}(p)$ au-dessus d'un point fermé $p\in W$ est également géométriquement rationnelle et satisfait la conjecture \hyperref[(Br-AF)]{(Br-AF)} sur $k(p)$ par \cite[Corollaire 8.13]{Sn}, donc elle satisfait la conjecture \hyperref[(E)]{(E)} par le Théorème \ref{TheoremLiang}. On conclut que $\hat X$ satisfait la conjecture \hyperref[(E)]{(E)} par le Théorème \ref{prop15} et satisfait le principe de Hasse lisse selon la~Proposition~\ref{core}.
\end{proof}
\subsection*{La conclusion}
\begin{theorem}\label{theoremeprincipal}
Soit $k$ un corps de nombres. Soit $X\subset \bbP^7_k$ une intersection complète géométriquement intègre non conique de deux quadriques. Alors tout modèle propre lisse de $X$ satisfait le principe de Hasse. 
\end{theorem}
\begin{proof}
On peut supposer que le pinceau de $X$ est non dégénéré par le Lemme~\ref{lemma0}(i),(ii). Si le rang de toute forme dans le pinceau géométrique de $X$ est au moins 7, on conclut par la Proposition \ref{casreg}. S'il existe une forme de rang au plus 5 on utilise la Proposition \ref{cas5}. Désormais on suppose qu'il existe une forme de rang 6 dans le pinceau géométrique de $X$. Chaque telle forme correspond à une racine d'ordre $\geqslant 2$ du polynôme caractéristique de $X$ par le Lemme \ref{lemma0}(iii). Ainsi, il y en a au plus 4. Le cas de 4 formes de rang 6 est traité dans la Proposition \ref{cas4formes}. Donc, on suppose qu'il y en a au plus 3. Si l'une des formes de rang 6 est définie sur $k$ le principe de Hasse lisse résulte de \cite[Theorem~9.5]{CSS}. Si l'une de ces formes est définie sur une extension quadratique $K/k$ on conclut par la Proposition \ref{cas2conj}. 

Il reste le cas où le pinceau géométrique de $X$ contient 3 formes de rang 6 dans une extension cubique $E/k$. Alors $X_E$ satisfait le principe de Hasse lisse sur $E$ par \cite[Theorem~9.5]{CSS}. Soit $\hat X$ un modèle propre lisse de $X$. Supposons que $\hat X(k_v)\neq \varnothing$ pour toute place $v$, donc $\hat X(E)\neq \varnothing$. Comme $X$ est une intersection de deux quadriques, elle possède un point rationnel lisse dans une extension $L/k$ de degré 4. (Il suffit de considérer l'intersection de $X$ avec un $k$-plan suffisamment générique.) On conclut que $Z_0^1(\hat X)\neq 0$ et $X(k)\neq \varnothing$ par le Corollaire \ref{corcycle-point}. Finalement, le principe de Hasse lisse résulte du Lemme~\ref{lemma0}(vi).
\end{proof}

\appendix
\section{Appendice}

Le but de cet appendice est de démontrer qu'une intersection régulière de deux quadriques dans $\bbP^7$ sur un corps de nombres $k$ contient un $k_v$-plan pour presque toute valuation $v$ sans utiliser le résultat de Wang \cite{Wn} (Proposition \ref{propplan}). On montre également que la condition de régularité est essentielle (Exemple \ref{exemplenonreg}).

On commence par un lemme sur la diagonalisation d'une forme bilinéaire symétrique sur un anneau de valuation discrète, voir aussi \cite[\S 19]{EKM}. Soit $A$ un anneau de valuation discrète tel que $2\in A^*$. On note $\pi$ une uniformisante de $A$ et $k=A/(\pi)$ le corps résiduel. Pour un $A$-module $V$ on utilise la notation $\bar V = V\otimes_A k$.

\begin{lemma}
 Soient $V$ un $A$-module libre de rang fini et $\varphi\colon V\otimes V\to A$ une forme bilinéaire symétrique. Alors il existe une $A$-base $\{v_i\}$ de $V$ telle que $\varphi(v_i,v_j)=0$ pour~$i\neq j$. 
\end{lemma}
\begin{proof} On raisonne par récurrence sur le rang de $V$. Supposons que la forme induite $\bar \varphi\colon \bar V\otimes \bar V\to k$ est non nulle. Alors il existe un élément $u\in V$ tel que $\varphi(u,u)\in A^*$. Soit $x\in V$, on définit
$$
x' = \frac{\varphi(x,u)}{\varphi(u,u)}u.
$$
Alors $x=x'+(x-x')$ et $\varphi(x-x',u)=0$. Ainsi, on a une décomposition orthogonale $V=\langle u\rangle\oplus \langle u\rangle^{\perp}$. Il suffit donc d'appliquer l'hypothèse à la restriction $\varphi|_{\langle u\rangle^{\perp}}$. 

Maintenant, on considère le cas où $\bar \varphi = 0$. Si la forme $\varphi$ est nulle l'énoncé est triviale. Sinon, on peut écrire $\varphi = \pi^k \varphi'$ avec $\bar \varphi'\neq 0$ et appliquer l'étape précédente à $\varphi'$.
\end{proof}

\begin{corollary}\label{cordiag}
Soient $V$ un $A$-module libre de rang fini et $F\in S^2(V^*)$ une forme quadratique. Supposons que la~forme $\bar F\in S^2(\bar V^*)$ est de rang $r$.  Alors il existe une $A$-base orthogonale de $V$ dans laquelle $F$ s'écrit comme 
$$F=\langle u_1\rangle\perp\langle u_2\rangle\perp\dots\perp\langle u_r\rangle\perp\langle\pi v_1\rangle\perp\langle\pi v_2\rangle\perp\dots \langle\pi v_n\rangle$$
où $v_i\in A$ et $u_i\in A^*$.
\end{corollary}

Soit $F$ un corps, $\car(F)\neq 2$. On associe à chaque forme quadratique son invariant de Clifford, ceci définit un homomorphisme $\cliff\colon I^2(F)\to \Br(F)[2]$ \cite[Theorem 14.3]{EKM} où $I^2(F)$ est le carré de l'idéal fondamental dans l'anneau de Witt, voir~\cite[\S4]{EKM}. 

Soit $A$ un anneau de valuation discrète tel que $2\in A^*$. Soient $K=\Frac(A)$ le corps de fractions et $k=A/(\pi)$ le corps résiduel. On rappelle que l'application de résidu $\res\colon \Br(K)\to H^1(k,\bbQ/\bbZ)$ \cite[1.4.3]{CTSk21} peut être définie au niveau des formes quadratiques \cite[Lemma 19.10]{EKM} et on a un diagramme commutatif
\begin{equation}\label{resdiag}
\begin{CD}
I^2(K) @>{\res}>> I^1(k) \\
@V{\cliff}VV @V{\disc}VV  \\
\Br(K)[2] @>{\res}>> k^*/(k^*)^2 
\end{CD}\end{equation}

On fixe la notation suivante. Soit $k$ un corps, $\car(k)\neq 2$ et soit $X\subset\bbP^7_k$ une intersection de deux quadriques donnée par des équations $F=G=0$. Supposons que le pinceau de $X$ est non dégénéré. Alors $\Psi=F+tG$ est une $k(t)$-forme quadratique non dégénérée, on lui associe l'invariant de Clifford $\alpha =\cliff(\Psi)\in \Br\big(k(t)\big)$ \cite[\S14]{EKM}. Quand la courbe $C\colon y^2=\disc(\Psi)$ est géométriquement irréductible on note $\tilde C$ sa normalisation et $k(C)$ le corps des fonctions rationnelles. Si $C$ est géométriquement réductible on a $\disc(\Psi)=aH(t)^2$ où $a\in k$ et $H(t)\in k[t]$ est un polynôme. Dans ce cas, on définit $k(C)=k(t)(\sqrt{a})$ et $\tilde C=\bbP^1_{k(\sqrt{a})}$. On note $\alpha_C\in Br\big(k(C)\big)$ la restriction de la classe~$\alpha$~à~$k(C)$. 

\begin{proposition}\label{prop1}
Dans la notation introduite ci-dessus, supposons que la variété $X$ contient un~plan sur $k$. Alors $\alpha_C=0\in \Br\big(k(C)\big)$. De plus, si $X$ contient une droite et $\alpha_C=0$ alors $X$ contient un plan.
\end{proposition}
\begin{proof}
Si $X$ contient un plan, par le Lemme \ref{lemme4} on peut écrire $\Psi=3H\perp T$ sur $k(t)$ où $H=\langle -1,1\rangle$ est un plan hyperbolique.
On~a 
$$\disc(\Psi)=\disc(3H\perp T)=\disc (T).$$ Comme la forme $T$ est de rang 2, elle est isotrope sur le corps $k(C)=k(\sqrt{\disc (T)})$. Ainsi sur $k(C)$ la forme $\Psi$ est purement hyperbolique et son invariant de Clifford $\alpha_C$ est nul.

Supposons que $X$ contient une droite et que $\alpha_C=0$. Par le Lemme \ref{lemme4} on décompose $\Psi=2H\oplus^{\perp} T$. On a $\disc(\Psi)=\disc (T)$, donc sur le corps $k(C)$ le discriminant de~la~forme~$T$ est un carré. On calcule l'invariant de Clifford en utilisant \cite[Lemma 14.2.]{EKM} $$\alpha_C=\cliff_{k(C)}(\Psi)=\cliff_{k(C)}(2H\oplus^{\perp} T)=\cliff_{k(C)}(2H)+\cliff_{k(C)}(T)=\cliff_{k(C)}(T).$$ Par un théorème de Pfister \cite[Théorème 8.1.1]{Kh} la forme $T$ est purement hyperbolique sur $k(C)$. Alors $T$ est isotrope sur $k(t)$ d'après \cite[Proposition 2.2]{CT} et on a une décomposition $\Phi=3H\oplus^{\perp} T'$, donc, l'énoncé résulte du théorème de Amer-Brumer \cite[Proposition~2.3]{CT}. 
\end{proof}

\begin{proposition}\label{prop2}
Soit $k$ un corps de nombres. Soit $X\subset \bbP^7_k$ une intersection complète géométriquement intègre de deux quadriques dont le pinceau est non dégénéré et ne contient pas de formes de rang $\leqslant 5$. Alors $X_{k_v}$ contient une droite sur $k_v$ pour presque toute place $v$ de $k$.
\end{proposition}
\begin{proof}
Soit $F_1\subset \Gr_1(\bbP^7_k)$ la variété des droites contenues dans $X$. Soit $\tilde F_1\subset F_1\times X$ la variété des couples $(\ell,P)$ où $\ell$ est une droite contenue dans $X$ est $P\in \ell$ est un point. On considère la projection $f\colon F_1\to X$. Pour un point $P\in X$ on note $Y_P\subset \bbP^4$ la base du cône $X\cap T_{X,P}$. Alors $Y_P$ est isomorphe à la fibre $f^{-1}(P)$. Par le Lemme~\ref{lemme6} il existe un ouvert $U\subset X$ tel que $Y_P$ est géométriquement intègre de dimension 2 pour tout point géométrique $P\in U$. Comme le morphisme $f$ est propre, cela implique que l'image réciproque $W=f^{-1}(U)$ est géométriquement intègre de dimension $5+2=7$. On note que l'autre projection $\pi:\tilde F_1\to F_1$ est un fibré projectif en droites. Ainsi, l'image $\pi(W)\subset F_1$ est une sous-variété géométriquement intègre et définie sur $k$ de dimension $\geqslant 7-1=6$. Donc, les estimations de Lang-Weil \cite{LW} assurent l'existence d'un $k_v$-point rationnel sur $\pi(W)$ pour presque toute place $v$.
\end{proof}

\begin{proposition}\label{prop4}
 Soit $k$ un corps de nombres et soit $X\subset\bbP^7_k$ une intersection complète géométriquement intègre de deux quadriques dont le pinceau est non dégénéré et ne~contient pas de formes de rang $\leqslant 5$. Alors les énoncés suivants sont équivalents.
 \begin{itemize}
     \item[(i)] la classe $\alpha_C$ appartient à~$\Br(\tilde C)\subset \Br\big(k(C)\big)$;
     \item[(ii)] $X_{k_v}$ contient un plan pour presque toute $v$.
 \end{itemize}
\end{proposition}
\begin{proof}
On note que (i) implique (ii). Effectivement, par la Proposition \ref{prop2} $X$~contient une droite pour presque toute $v$. La classe $\alpha_C\in Br(\tilde C)$ s'annule pour presque toute $v$ d'après \cite[Theorem 10.3.1]{CTSk21} et on conclut par la Proposition \ref{prop1}.
 
 Montrons que (ii) implique (i). Par les Propositions \ref{prop1}, \ref{prop2}, $\alpha_C$ s'annule dans $\Br\big(k_v(C)\big)$ pour presque toute $v$. L'application de résidu nous fournit un diagramme commutatif pour chaque place~$v$ et chaque point fermé $x\in \tilde C $ 
\begin{equation*}
\begin{CD}
1 @>>> \Br(\tilde C) @>>> \Br\big(k(C)\big) @>{\res_x}>> \displaystyle{ H^1\big(k(x),\bbQ/\bbZ\big)} \\
@. @VVV @VVV @VVV \\
1 @>>> \Br(\tilde C_v) @>>> \Br\big(k_v(C)\big) @>{\res_x}>> \displaystyle{ \bigoplus_{w|v} H^1\big(k(x)_w,\bbQ/\bbZ\big)}  
\end{CD}\end{equation*}
où la deuxième somme est pour chaque valuation $w$ de $k(x)$ au dessus de la valuation $v$ de $k$. Les suites horizontales sont des complexes de groupes.
 
Supposons que $\alpha_C\notin Br(\tilde C)$. Par \cite[Theorem 3.6.1]{CTSk21} il existe un point fermé $x\in \tilde C$ tel que le résidu $\beta=\res_x(\alpha_C)\in H^1\big(k(x),\bbQ/\bbZ\big)$ est non trivial. On déduit du diagramme ci-dessus que $\beta_w\in H^1\big(k(x)_w,\bbQ/\bbZ\big)$ est nulle pour presque toute valuation $w$ de $k(x)$. Ceci contredit le théorème de Tchebotariov qui assure qu'un nombre infini de places $w$ sont inertes dans l'extension cyclique de $k(x)$ définie 
 par $\beta$.
\end{proof}

\begin{proposition}\label{propplan}
Soit $k$ un corps de nombres et soit $X\subset\bbP^7_k$ une intersection complète géométriquement intègre de deux quadriques dont le pinceau est non dégénéré et ne~contient pas de formes de rang $\leqslant 6$. Alors $X_{k_v}$ contient un plan pour presque toute place $v$ de $k$.
\end{proposition}
\begin{proof}
Selon la Proposition \ref{prop4} il suffit de vérifier que $\alpha_C\in\Br(\tilde C)$. 
Par \cite[Theorem 3.6.1]{CTSk21}, on  doit montrer que le résidu $\res_x (\alpha_C)$ est nul pour tout point fermé $x\in \tilde C$. Soit $s\in \bbP^1_k$ l'image de $x$ sous la projection $\tilde C\to \bbP^1_k$ et soit $A=\mathcal{O}_{\bbP^1_k,s}$ l'anneau locale en $s$.
Par le Corollaire \ref{cordiag} on peut diagonaliser $\Psi$ sur $k(t)$ comme 
$$\Psi \simeq_{\,k(t)} \langle u_1\rangle\perp\langle u_2\rangle\perp\dots\perp\langle u_7\rangle\perp\langle v\rangle$$
où $u_i\in A^*$ sont inversibles. Soit $u=\prod_{i} u_i$. Dans $k(C)$ on a $y^2=uv$, $v=(y/u)^2u$. Alors la forme $\Psi$ s'écrit comme $$\Psi \simeq_{\,k(C)} \langle u_1\rangle\perp\langle u_2\rangle\perp\dots\perp\langle u_7\rangle\perp\langle u\rangle$$
et on conclut que $\res_x(\Psi)=0$ puisque tous les coefficients sont inversibles. On déduit du diagramme \eqref{resdiag} que $\res_x(\alpha_C)=0$.
\end{proof}

\begin{example}\label{exemplenonreg} Il est facile de construire un exemple d'une intersection de deux quadriques dont le pinceau contient une forme de rang $6$ telle que $X_{k_v}$ ne contient pas de plan pour un nombre infini de places $v$. Soit $k$ un corps, $\car(k)\neq 2$. Soit $X$ une intersection de deux quadriques définie par deux $k$-formes diagonales  
\begin{equation*}
\begin{split}
    F=a_2x_2^2+a_3x_3^2+a_4x_4^2+a_5x_5^2+a_6x_6^2+a_7x_7^2, \\
    G=b_0x_0^2+b_1x_1^2+b_2x_2^2+b_3x_3^2+b_4x_4^2+b_5x_5^2+b_6x_6^2+b_7x_7^2
\end{split}
\end{equation*}
où tous $a_i,b_j$ sont non nuls. Ainsi, $\rk(F)=6$ et $\rk(G)=8$. La forme $\Psi$ s'écrit comme 
\begin{equation}\label{psiform}
\Psi=\langle b_0 t,\,b_1 t,\, a_2+b_2t, a_3+b_3t, a_4+b_4t, a_5+b_5t, a_6+b_6t, a_7+b_7t.\rangle\end{equation}
On a $\disc(\Psi)=t^2 P(t)$ où $P(t)=\prod_{i=2}^7(a_i+tb_i)$. 
Dans une voisinage du point $t=0$ la~normalisation de la courbe $C\colon y^2=\disc(\Psi)$ est de la forme $\tilde C\colon y^2=P(t)$. 
On déduit de \eqref{psiform} que $\res_0(\Psi)=\langle b_0,\,b_1 \rangle\in I^1W(k)$. Par le diagramme \eqref{resdiag} on trouve le résidu de la classe $\alpha$
$$\res_0(\alpha)=\disc \langle b_0,\,b_1 \rangle = -b_0b_1\in k^*/(k^*)^2.$$

On considère un point $x \in \tilde C$ au dessus de $t=0$. Soit $K=k\big(\sqrt{P(0)}\big)$ le~corps de~résidu~en~$x$. Si on a $\alpha_C\in Br(\tilde C)$, alors le résidu $\res_x(\alpha_C)$ est nul et $-b_0b_1\in (K^*)^2$. Donc, $-b_0b_1$ est un carre dans $k$ ou $-b_0b_1P(0)$ l'est. Il est facile de trouver un exemple des coefficients $a_i, b_j$ tels qu'aucune de ces deux condition n'est satisfaite. Alors $X_{k_v}$ ne contient pas de plan pour un nombre infini de places $v$ par la Proposition \ref{prop4}.
\end{example}

\begin{remark}
Soit $X\subset \bbP^7$ est une intersection de deux quadriques dont le pinceau contient une quadrique $Q$ de rang 6. Soit $L$ la droite des sommets de $Q$. Alors la projection depuis $L$ envoie la variété $F_2$ des plans contenus dans $X$ dans la variété des plans contenus dans une quadrique lisse $Q'\subset \bbP^5$. Cette dernière variété a deux composantes connexes \cite[Theorem 1.2]{R}. Cette observation suggère que $F_2$ n'est pas géométriquement intègre et donc peut ne pas avoir de point lisse $k_{v}$-rationnel pour une infinité de places $v$.
\end{remark}

{\small{\scshape Département de mathématiques et applications, École normale supérieure, \\45 rue d’Ulm, Paris, France 75005}}
\\
\textit{Email address:} \href{mailto:alexander.molyakov@ens.psl.eu}{\texttt{alexander.molyakov@ens.psl.eu}}


\begin{thebibliography}{9}
\small

\bibitem[Br]{Br} N. Bourbaki, Algèbre, Chapitres 4 à 7. Paris, Masson, 1981 (1re éd. (Hermann) 1950-1952), 422 p.

\bibitem[CF]{CF} Algebraic Number Theory: Proceedings of an Instructional Conference. Edited by J. W. S. Cassels and A. Fröhlich.

\bibitem[CT22]{CT} J.-L. Colliot-Thélène, Retour sur l'arithmétique des intersections de deux quadriques, avec un appendice par A. Kuznestov, arXiv 2208.04121 (2022).

\bibitem[CT03]{CT03} J.-L. Colliot-Thélène, Points rationnels sur les fibrations, Higher dimensional varieties and rational points (Budapest, 2001). Bolyai Soc. Math. Stud., vol. 12, Springer, Berlin, 2003, pp. 171–221.

\bibitem[CT95]{CT95} J.-L. Colliot-Thélène, L’arithmétique du groupe de Chow des zéro-cycles. J. Théorie des nombres de Bordeaux 7 (1995), 51–73.

\bibitem[CTSn]{CTSn} J.-L. Colliot-Thélène, J.-J. Sansuc, On the Chow groups of certain rational surfaces: complements to a paper of. S. Bloch. Duke Math. J. 48(2): 421-447 (June 1981).

\bibitem[CSS]{CSS} J.-L. Colliot-Thélène, J.-J. Sansuc, Sir Peter Swinnerton-Dyer, Intersections of two quadrics and Châtelet surfaces. Journal für die reine und angewandte Mathematik (Crelle). I, Bd. 373(1987) 37-107; II Bd. 374(1987) 72-168.

\bibitem[CTSk21]{CTSk21} J.-L. Colliot-Thélène, A. Skorobogatov, The Brauer–Grothendieck group. Ergebnisse der Mathematik und ihrer Grenzgebiete. 3. Folge, 71. Springer, Cham, 2021

\bibitem[CTSk92]{CTSk} J.-L. Colliot-Thélène, A. Skorobogatov, Approximation faible pour les intersections de deux quadriques en dimension 3. C.R. Acad. Sci. Paris , t. 314, Série I, p.127-132 (1992).

\bibitem[CrTs]{CorTs} D. F. Coray, M. A. Tsfasman, Arithmetic on singular Del Pezzo surfaces. Proc. London Math. Soc. (3) 57 (1988) 25–87.

\bibitem[EKM]{EKM}
R. Elman, N. Karpenko, A. Merkurjev,
The Algebraic and Geometric Theory of Quadratic Forms.
American Mathematical Society, Colloquium Publications, Vol. 56, Providence, (2008).


\bibitem[Hr]{Har} D. Harari, Méthode des fibrations et obstruction de Manin. Duke Math. J. 75 (1994), no. 1, 221–260.

\bibitem[HW]{HW} Y. Harpaz, O. Wittenberg, On the fibration method for zero-cycles and rational points. Annals of Mathematics, Volume 183 (2016) p. 229-295.

\bibitem[HB]{HB} R. Heath-Brown, Zeros of pairs of quadratic forms. Journal für die reine und angewandte Mathematik (Crelle) 739 (2018), 41–80.

\bibitem[Kh]{Kh} B. Kahn, Formes quadratiques sur un corps. Cours spécialisés 15, Société mathématique de France (2008).

\bibitem[KS]{KS} K. Kato, S. Saito, Global class field theory of arithmetic schemes, Applications of algebraic K-theory to algebraic geometry and number theory, Part I (Boulder, Colo., 1983). Contemp. Math., vol. 55, Amer. Math. Soc., Providence, RI, 1986, pp. 255–331.

\bibitem[LW]{LW} S. Lang, A. Weil, Number of Points of Varieties in Finite Fields. American Journal of Mathematics, Vol. 76, No. 4 (Oct., 1954), pp. 819-827.

\bibitem[Ln]{L} Y. Liang, Arithmetic of 0-cycles on varieties defined over number fields. Annales scientifiques de l'École Normale Supérieure, Série 4, Tome 46 (2013) no. 1, pp. 35-56.

\bibitem[Mn]{Mn} Yu. I. Manin, Le groupe de Brauer–Grothendieck en géométrie diophantienne. Actes du Congrès International des Mathématiciens (Nice, 1970), Tome 1, Gauthier-Villars, Paris, 1971, pp. 401–411.

\bibitem[Sl93]{Sb} P. Salberger, On the arithmetic of intersections of two quadrics containing a conic, 	arXiv 2305.02289 (1993).

\bibitem[Sl88]{S88} P. Salberger, Zero-cycles on rational surfaces over number fields. Inventiones mathematicae volume 91, 505–524 (1988).

\bibitem[SbSk]{SbSk} P. Salberger, A. N. Skorobogatov, Weak approximation for surfaces defined by two quadratic forms. Duke Math. J. 63 (1991), no. 2, 517–536. 

\bibitem[Sn]{Sn} J.-J. Sansuc,
Groupe de Brauer et arithmétique des groupes algébriques linéaires sur un corps de nombres. Journal für die reine und angewandte Mathematik (Crelle) 327. 

\bibitem[Sr]{S} J.-P. Serre, Lie Algebras and Lie Groups.  Springer Berlin Heidelberg; 2nd ed. (1992). 

\bibitem[Pn]{PY} B. Poonen, Rational Points on Varieties. Graduate Studies in Mathematics,Volume 186, 337 pp (2017).

\bibitem[Rd]{R} M. Reid, The complete intersection of two or more quadrics. Thesis, Trinity College, Cambridge, June 1972.

\bibitem[Wn]{Wn} X. Wang, Maximal linear spaces contained in the base loci of pencils of quadrics. Algebraic Geometry 5 (3) (2018) 359–397.


\end{thebibliography}
\end{document}